\documentclass[oneside,reqno]{amsart}
\usepackage{etex}
\usepackage[a4paper]{geometry}
\usepackage{amssymb,amsfonts,amsmath}
\usepackage{comment} 
\usepackage[all,arc]{xy}
\usepackage{enumerate}
\usepackage{mathrsfs,mathtools}
\usepackage{todonotes,booktabs}
\usepackage{stmaryrd}
\usepackage{marvosym}
\usepackage{graphicx}
\usepackage{pdflscape}
\usepackage{bbm}
\usepackage{tikz}
\usepackage{tikz-cd}
\usetikzlibrary{trees}
\usetikzlibrary[shapes]
\usetikzlibrary[arrows]
\usetikzlibrary{patterns}
\usetikzlibrary{fadings}
\usetikzlibrary{backgrounds}
\usetikzlibrary{decorations.pathreplacing}
\usetikzlibrary{decorations.pathmorphing}
\usetikzlibrary{positioning}
\def\biblio{\bibliography{duality}\bibliographystyle{alpha}}
\usepackage{xcolor} %improved colors
\usepackage{graphicx}

\usepackage[pagebackref]{hyperref}

\definecolor{dark-red}{rgb}{0.5,0.15,0.15}
\definecolor{dark-blue}{rgb}{0.15,0.15,0.6}
\definecolor{dark-green}{rgb}{0.15,0.6,0.15}
\hypersetup{
    colorlinks, linkcolor=dark-red,
    citecolor=dark-blue, urlcolor=dark-green
}

\renewcommand*{\backref}[1]{}
\renewcommand*{\backrefalt}[4]{%
  \ifcase #1 %
No citations.% use \relax if you do not want the "No citations" message
  \or
(cit. on p. #2).%
  \else
(cit on pp. #2).%
  \fi%
}
\usepackage{subfiles}

\usepackage[nameinlink,capitalise,noabbrev]{cleveref}

\newtheorem{thm}{Theorem}[section]
\newtheorem{cor}[thm]{Corollary}
\newtheorem{prop}[thm]{Proposition}
\newtheorem{lem}[thm]{Lemma}

\theoremstyle{definition}
\newtheorem{defn}[thm]{Definition}

\newtheorem{ex}[thm]{Example}

\theoremstyle{remark}
\newtheorem{rem}[thm]{Remark}

\theoremstyle{theorem}
\newtheorem*{thm*}{Theorem}
\newtheorem*{cor*}{Corollary}
\newtheorem*{prop*}{Proposition}
\newtheorem*{defn*}{Definition}

\makeatletter
\let\c@equation\c@thm
\makeatother
\numberwithin{equation}{section}

\DeclareMathOperator{\Sp}{Sp}
\DeclareMathOperator{\Hom}{Hom}
\DeclareMathOperator{\RHom}{RHom}
\DeclareMathOperator{\End}{End}

\DeclareMathOperator{\cE}{\mathcal{E}}

\DeclareMathOperator{\cF}{\mathcal{F}}
\DeclareMathOperator{\cG}{\mathcal{G}}

\DeclareMathOperator{\Ext}{Ext}

\DeclareMathOperator{\Spec}{Spec}
\DeclareMathOperator{\Mod}{Mod}

\DeclareMathOperator{\StMod}{StMod}

\DeclareMathOperator{\Loc}{Loc}

\newcommand{\Q}{\mathbb{Q}}

\newcommand{\kos}[2]{{#1}/\!\!/{#2}}

\newcommand{\cal}{\mathcal}
\newcommand{\xr}{\xrightarrow}

\newcommand{\Z}{\mathbb{Z}}

\Crefname{figure}{Figure}{Figures}
\Crefname{assu}{Assumption}{Assumptions}
\Crefname{lem}{Lemma}{Lemmas}
\Crefname{thm}{Theorem}{Theorems}

\renewcommand{\frak}{\mathfrak}

\DeclareMathOperator{\Cell}{Cell}

\DeclareMathOperator{\Inj}{Inj}

\newcommand{\fp}{\mathfrak{p}}

\newcommand{\recollement}[5]{
\xymatrix{{#1} \ar[r]|-{#2} & #3 \ar[r]|-{#4} \ar@<1ex>[l]^-{{#2}_!} \ar@<-1ex>[l]_-{{#2}^*} & #5, \ar@<1ex>[l]^-{{#4}!} \ar@<-1ex>[l]_-{{#4}^*}
}}
\let\lim\relax

\DeclareMathOperator{\lim}{lim}

\newcommand{\F}{\mathbb{F}}

\newcommand{\cL}{\mathcal{L}}

\DeclareMathOperator{\Map}{Map}

\title{Local Gorenstein duality for cochains on spaces}
\author{Tobias Barthel}
\address{Max-Planck-Institut f\"ur Mathematik, Vivatsgasse 7,
53111 Bonn,
Germany}
\email{tbarthel@mpim-bonn.mpg.de}

\author{Nat{\`a}lia Castellana}
\address{Departament de Matem\`atiques, Universitat Aut\`onoma de Barcelona, 08193 Bellaterra, Spain, and BGSMATH}
\email{natalia@mat.uab.cat}
\author{Drew Heard}
\address{Department of Mathematical Sciences, Norwegian University of Science and Technology, Trondheim}
\email{drew.k.heard@ntnu.no}
\author{Gabriel Valenzuela}
\address{Max-Planck-Institut f\"ur Mathematik, Vivatsgasse 7,
53111 Bonn,
Germany}
\email{gvalenzuela@mpim-bonn.mpg.de}

\date{\today}
\begin{document}

\begin{abstract}
We investigate when a commutative ring spectrum $R$ satisfies a homotopical version of local Gorenstein duality, extending the notion previously studied by Greenlees. In order to do this, we prove an ascent theorem for local Gorenstein duality along morphisms of $k$-algebras.  Our main examples are of the form $R = C^*(X;k)$, the ring spectrum of cochains on a space $X$ for a field $k$. In particular, we establish local Gorenstein duality in characteristic $p$ for $p$-compact groups and $p$-local finite groups as well as for $k = \Q$ and $X$ a simply connected space which is Gorenstein in the sense of Dwyer, Greenlees, and Iyengar.
\end{abstract}
%Version 1:
%\begin{abstract}
%We investigate when a ring spectrum $R$ satisfies a localized version of the Gorenstein duality condition considered by Greenlees. In order to do this, we prove an ascent theorem for local Gorenstein duality for a finite morphism of $k$-algebras.  Our main examples are of the form $R = C^*(X;k)$, the ring spectrum of cochains on a space $X$ for a field $k$. When $k = \Q$ we prove that if $X$ is a simply connected space with Noetherian rational cohomology, then if $C^*(X;\Q) \to \Q$ is Gorenstein in the sense of Dwyer, Greenlees, Iyengar, it also satisfies local Gorenstein duality. When $k = \F_p$ we consider the case where $X$ is the classifying space of a $p$-local compact group in the sense of Broto, Levi, and Oliver, and show that local Gorenstein duality is satisfied if $X$ is the classifying space of a $p$-compact group or of a $p$-local finite group. 
%\end{abstract}

\maketitle

%{\hypersetup{linkcolor=black}\tableofcontents}
\setcounter{tocdepth}{1}
\tableofcontents
\def\biblio{}
\section{Introduction}
Given a Noetherian commutative local ring $(A,\frak{m},k)$, there are numerous equivalent conditions for when $A$ is Gorenstein. In particular, if $A$ has Krull dimension $n$, then $A$ is Gorenstein if and only if 
\[
\Ext_A^i(k,A) \cong \begin{cases}
	k & i = n \\
	0 & \text{otherwise.}
\end{cases}
\]
In the derived category $D(A)$, this can be restated in terms of the derived hom as an equivalence $\RHom_A(k,A) \simeq \Sigma^nk$. Inspired by this, Dwyer, Greenlees, and Iyengar \cite{dgi_duality} introduced the notion of a Gorenstein ring spectrum. More specially, if $k$ is a field and $R$ a commutative ring spectrum,  a morphism $R \to k$ of ring spectra\footnote{If $k$ is a field, we will also denote by $k$ the Eilenberg--MacLane spectrum $Hk$.} (always assumed to be commutative) is said to be Gorenstein of shift $r$ if there is an equivalence $\Hom_R(k,R) \simeq \Sigma^r k$ for some integer $r$. 
 
 One is particularly interested in the duality that the Gorenstein condition implies. Assume that $R$ is a $k$-algebra. If $R \to k$ is Gorenstein then, under some additional orientable hypothesis and coconnectiveness, $R$ automatically satisfies the property that $\Cell_k(R) \simeq \Sigma^r\Hom_k(R,k)$, where $\Cell_k$ is the $k$-cellular approximation to $R$, see \Cref{sec:gorduality}. As we shall see, this is the analog of the classical characterization of Gorenstein rings as those commutative local Noetherian rings $A$ of Krull dimension $n$ for which the local cohomology with respect to $\frak m$ satisfies 
 \[
H^i_{\frak m}(A) \cong \begin{cases}
	I_{\frak m} & i = n \\
	0 & \text{otherwise,}
\end{cases}
 \]
where $I_{\frak m} \cong \Hom_k(A,k)$ denotes the injective hull of $k$. Whenever the equivalence $\Cell_k(R) \simeq \Sigma^r\Hom_k(R,k)$ is satisfied for a morphism $R \to k$ of ring spectra, we say that $R$ satisfies Gorenstein duality. Note however that, in contrast to the algebraic situation, $R \to k$ being Gorenstein does not imply that $R$ satisfies Gorenstein duality, see \Cref{rem:gorenstein_not_duality}. The structural implications for $\pi_*R$ when $R$ satisfies Gorenstein duality have been investigated previously by Greenlees and Lyubeznik \cite{green_lyub}. 

Our examples will mostly come from ring spectra of the form $R = C^*(X;k)$, the ring spectrum of $k$-valued cochains on a suitable space $X$. If $X = BG$ where $G$ is a finite group, then $C^*(BG;\F_p) \to \F_p$ is always Gorenstein of shift $0$, and $C^*(BG;\F_p)$ satisfies Gorenstein duality of the same shift, even though the cohomology ring $\pi_{-*}C^*(BG;\F_p) \cong H^*(BG;\F_p)$ is not Gorenstein in general \cite[Section 10.3]{dgi_duality}. A consequence of the structural implications mentioned earlier is that $H^*(BG;\F_p)$ is Cohen--Macaulay if and only if $H^*(BG;\F_p)$ is Gorenstein, a result originally shown by Benson and Carlson \cite{BensonCarlson1994Projective}. 
 
 If $A$ is a commutative local Gorenstein ring, then the localization $A_{\frak p}$ at any prime ideal $\fp \in \Spec(A)$ is still Gorenstein. One way to see this is to use yet another interpretation of Gorenstein rings as those commutative local rings with finite injective dimension as an $A$-module. Alternatively, we observe that if $\frak p$ has dimension $d$, then the ring $A_{\frak p}$ is local Noetherian of dimension $n-d$, and Greenlees--Lyubeznik's dual localization \cite[Section 2]{green_lyub} can be used to show that 
 \begin{equation}\label{eq:local_gorenstein}
 H_{\frak p}^{*}(A_{\frak p}) \cong I_{\frak p}[n-d] 
 \end{equation}
% \begin{equation}\label{eq:local_gorenstein}
% H_{\frak p}^{i}(A_{\frak p}) \cong \begin{cases}
% 	I_{\frak p} & i = n-d \tag{$*$} \\
% 	0 & \text{otherwise,}
% \end{cases}
% \end{equation}
 where $I_{\frak p}$ is the injective hull of $A/\frak{p}$,\footnote{ We note that by \cite[Proposition 3.77]{lambook}, $I_{\frak p}$ has a natural structure as an $A_{\frak p}$-module, and is isomorphic to the injective hull of $A_{\frak p}/\frak p$.} and hence $A_{\frak p}$ is still Gorenstein. 
 
 Now suppose that $R$ is a ring spectrum. As we will explain below, for any $\frak p \in \Spec^h(\pi_*R)$, we can form spectral versions of the terms in the equation  (\ref{eq:local_gorenstein}); that is, $R_{\frak p}$ for localization at $\frak p$, $\Gamma_\frak p R$ for local cohomology of $R_{\frak p}$ and $T_R(I_{\frak p})$ for injective hulls, where $I_{\frak p}$ is the injective hull of $(\pi_*R)/\frak p$ and $\pi_*(T_R(I_{\frak p}))\cong I_{\frak p}$. We thus say that $R \to k$ satisfies local Gorenstein duality of shift $a$ if the spectral version of \eqref{eq:local_gorenstein} holds for all $\frak p$, 
 \[
 \Gamma_{\frak p}R \simeq \Sigma^{a+d}T_R(I_{\frak p}).
 \] 
An outcome of this equivalence is the fact that the homotopy of the spectrum $\Gamma_{\frak p}R$ is determined by algebraic information in $\pi_*(R)$, see \Cref{thm:green_lyub,rem:rickard} This spectral version was introduced by Barthel, Heard and Valenzuela in \cite[Definition 4.21]{bhv2} under the terminology of absolute Gorenstein duality. For a formal definition of local Gorenstein duality, see \Cref{sec:localgorenstein}. The consequences of local Gorenstein duality for the ring $\pi_*R$ are reviewed in \Cref{thm:green_lyub} below. For example, it implies that the ring $\pi_*R$ is generically Gorenstein, i.e., the localization of $\pi_*R$ at any minimal prime ideal is Gorenstein.  In fact, the main result of \cite{bg_localduality} is that $C^*(BG;\F_p) \to \F_p$ satisfies local Gorenstein duality when $G$ is a finite group, or more generally for certain compact Lie groups with an orientability property. In this context the modules $\Gamma_{\frak p}(R)$ are understood as generalizations of the idempotent Rickard modules in the stable module category $\StMod_G$ when $G$ is a finite group (see \cite[Theorem 2]{{benson_shortproof}}), which were first used to classify the thick subcategories of the compact objects of $\StMod_G$ \cite{BensonCarlsonRickard1997Thick}, see \Cref{rem:rickard}.  
 
A fundamental difference between the algebraic and topological situations is that in topology we do not know in general that Gorenstein duality implies local Gorenstein duality. The first objective for this work is to identify conditions where local Gorenstein duality holds.
The main techniques to determine whether a ring spectrum $R \to k$ is Gorenstein are the Gorenstein ascent and descent theorems of Dwyer--Greenlees--Iyengar, see \cite[Section 19]{greenlees_hi} for a summary, and  \cite{bhv2} for ascent techniques in local Gorenstein duality. 

Inspired by the Gorenstein ascent of Dwyer--Greenlees--Iyengar, we prove the following ascent theorem for local Gorenstein duality. 
\begin{thm*}[\Cref{cor:localdualityfinite}]
	Let $S \xr{f} R$ be a finite morphism of augmented $k$-algebras and $Q = R \otimes_S k$. Assume that the following conditions are satisfied:
	\begin{enumerate}
	 	\item $R \to k$ and $S \to k$ are orientable Gorenstein (\Cref{def:orientable}) of shift $r$ and $s$ respectively.
	 \item $Q \to k$ is cosmall, i.e., $Q$ is in the thick subcategory in $\Mod_Q$ generated by $k$. 
	 \item $R$ and $S$ are dc-complete (\Cref{sec:gorduality}).
	 \end{enumerate}  
	 Then, if $S$ satisfies local Gorenstein duality of shift $s$, then $R$ satisfies local Gorenstein duality of shift $r$. 
\end{thm*}

\begin{comment}
\subsection*{Local Gorenstein duality}
We now briefly explain how to generalize the local Gorenstein condition \eqref{eq:local_gorenstein} to a ring spectrum $R$. More details are given in \Cref{sec:localgorenstein}. To begin, we recall that given a commutative ring spectrum $R$ and a homogeneous prime ideal $\frak p\in \Spec^h(\pi_*R)$, one can form a version of the localization at $\frak p$ to produce a new spectrum $R_{\fp}$ with $\pi_*(R_{\frak p}) \cong (\pi_*R)_{\fp}$. Moreover, there is a functor $\Gamma_{\cal{V}(\frak p)}$ such that the homotopy $\pi_*(\Gamma_{\cal V(\frak p)}R_{\frak p})$ is the target of a spectral sequence whose $E_2$-term is $H^{s,t}_{\frak p}((\pi_*R)_{\frak p})$. We hence take $\Gamma_{\frak p}R = \Gamma_{\cal V(\frak p)}(R_{\frak p})$ as our spectral analog of the left hand side of \eqref{eq:local_gorenstein}. For the right hand side of \eqref{eq:local_gorenstein} we recall that given any injective $\pi_*R$-module $I$, one can form an $R$-module spectrum $T_R(I)$, such that $\pi_*T_R(I) \cong I$, see \Cref{sec:localgorenstein}. We are then led to the following definition. 
\begin{defn*}
	A ring spectrum $R$ satisfies local Gorenstein duality (with shift $a$) if for each $\fp\in \Spec(\pi_*R)$ of dimension $d$, there is an equivalence $\Gamma_{\frak p}R \simeq \Sigma^{a+d}T_R(I_{\frak p})$, where $I_{\frak p}$ is the injective hull of $(\pi_*R)/\frak p$.
\end{defn*}
\end{comment}
%\subsection*{Examples}
Our main examples come from ring spectra of the form $C^*(X;k)$ with emphasis on classifying spaces of compact Lie groups and its homotopical generalizations. Here the technical conditions of orientability and dc-completeness are satisfied automatically when $X$ is a suitably nice space. In particular, they are satisfied for spaces of Eilenberg--Moore type (EM-type), see \Cref{def:emtype}. The previous proposition specializes to the following statement. 
\begin{thm*}[\Cref{thm:relgorensteinspaces}]
		Let $g \colon Y \to X$ be a morphism of spaces of EM-type ($p$-complete if the characteristic of $k$ is $p$) with fiber $F$.  Suppose that $H^*(F;k)$ is finite-dimensional, and that $C^*(X;k)$ is Gorenstein of shift $s$.
\begin{enumerate}
  \item If $C^*(Y;k)$ is Gorenstein of shift $r$, then $g^* \colon C^*(X;k) \to C^*(Y;k)$ is relatively Gorenstein of shift $s-r$.
  \item If, in addition, $C^*(X;k)$ satisfies local Gorenstein duality of shift $s$, then $C^*(Y;k)$ satisfies local Gorenstein duality of shift $r$.
\end{enumerate}
\end{thm*}
Let $G$ be a compact Lie group such that (if $p >2$) the adjoint representation of $G$ is orientable, and consider a unitary embedding $f \colon G \to U(n)$. Taking $g = Bf^{\wedge}_p$ in the above theorem, one recovers the result of Benson and Greenlees that $C^*(BG;\F_p) \to \F_p$ satisfies local Gorenstein duality of shift $\dim(G)$. The same result (with no orientability hypothesis) holds for the classifying space of a $p$-compact group $\cal{G}$, that is, $C^*(B\cal{G},\F_p) \to \F_p$ satisfies local Gorenstein duality of shift $\dim_p(\cal{G})$, where $\dim_p(X)$ denotes the $\F_p$-cohomological dimension of a space $X$.

It is not true in general that if $G$ is a compact Lie group, then $C^*(BG;\F_p)$ satisfies local Gorenstein duality of shift $\dim(G)$. A simple example described in \cite{greenlees_borel} is given by $G = O(2)$ at an odd prime, which satisfies local Gorenstein duality of shift 3, while $\dim(G) = 1$. However, the triple $(\Omega (BO(2)^{\wedge}_p)),BO(2)^{\wedge}_p,\text{id})$ is a $p$-compact group, even though $\Omega (BO(2)^{\wedge}_p) \not \simeq O(2)^{\wedge}_p$. In fact $\Omega (BO(2)^{\wedge}_p)\simeq \Omega (B(S^3)^{\wedge}_p)\simeq (S^3)^{\wedge}_p$ at odd primes, which explains the shift obtained as the $\F_p$-cohomological dimension. More generally, we say that a compact Lie group is of $p$-compact type if the triple $(\Omega (BG^{\wedge}_p),BG^\wedge_p,\text{id})$ is a $p$-compact group (this is true if and only if $\Omega (BG^{\wedge}_p)$ is $\F_p$-finite).  In this case, $C^*(BG;\F_p) \to \F_p$ will satisfy Gorenstein duality of shift $\dim_p (\Omega (BG^{\wedge}_p))$ by the previous result on $p$-compact groups. Taking $G = O(2)$ at an odd prime as above, we have $\dim_p (\Omega (BO(2)^{\wedge}_p)) = 3$, as expected. 

A common generalization of both $p$-compact groups and compact Lie groups is the notion of a $p$-local compact group of Broto, Levi, and Oliver \cite{blo_pcompact}. Given such a $p$-local compact group $\cG$, there is an associated classifying space $B\cal{G}$, and one can ask if $C^*(B\cal{G};\F_p) \to \F_p$ satisfies local Gorenstein duality, or even if it is Gorenstein. We do not know the answer to this question in full generality, however we identify conditions for this to occur in \Cref{sec:pcompact}. In the case of $p$-local finite groups \cite{blo_fusion} we deduce from work of Cantarero, Castellana, and Morales \cite{ccm_vbplocalfinite}, that $C^*(B\cal{G};\F_p) \to \F_p$ satisfies local Gorenstein duality of shift $0$. In summary, we obtain the following results. 

\begin{thm*}[\Cref{thm:localgorensteindualitypcompact,thm:liepcompact,cor:plocalfinitegorenstein}]
	Let $\cal{G}$ be a $p$-local compact group of one of the following types:
	\begin{enumerate}
		\item associated to a Lie group of $p$-compact type,
		\item $\Omega(B\cal{G}^\wedge_p)$ is a $p$-compact group, or
		\item a $p$-local finite group. 
	\end{enumerate}  Then, $C^*(B\cal{G},\F_p) \to \F_p$ satisfies local Gorenstein duality of shift $\dim_p(\Omega(B\cG^{\wedge}_p))$ (Cases (1) and (2)) or $0$ (Case (3)), respectively.
\end{thm*}

For example, this shows that $O(2n)$ at an odd prime satisfies local Gorenstein duality of shift $n(2n+1)$, while $\dim(O(2n)) = n(2n-1)$. 

In the rational case, a stronger result holds. Using the fact that algebraic Noether normalization can be lifted to the spectrum level, we show that for any simple connected rational space with Noetherian cohomology, Gorenstein implies local Gorenstein duality of the same shift. 
\begin{thm*}[\Cref{thm:rational}]
	Let $X$ be a simply connected rational space with Noetherian cohomology. If $C^*(X;\Q) \to \Q$ is Gorenstein of shift $r$, then $C^*(X;\Q)$ satisfies local Gorenstein duality of shift $r$. 
\end{thm*}

\subsection*{Convention}

Throughout this document, all rings and structured ring spectra will be assumed to be commutative, which in the case of commutative ring spectra means carrying an $E_{\infty}$-ring structure.
\subsection*{Acknowledgements}

We would like to thank the CRM for funding a visit of the authors through the Research in Pairs program. The first author was partially supported by the DNRF92 and the European Unions Horizon 2020 research and innovation programme under the Marie Sklodowska-Curie grant agreement No.~751794, the second author was partially supported by FEDER-MEC grant MTM2016-80439-P and acknowledges financial support from the Spanish Ministry of Economy and Competitiveness through the  ``Mar\'ia de Maeztu'' Programme for Units of Excellence in R\&D (MDM-2014-0445), the third author was partially supported by the SFB 1085 ``Higher Invariants", and the first and fourth author would like to thank the Max Planck Institute for Mathematics for its hospitality. The first and second author would furthermore like to thank the Isaac Newton Institute for Mathematical Sciences, Cambridge, for support and hospitality during the programme \emph{Homotopy Harnessing Higher Structures}, where work on this paper was undertaken. This work was supported by EPSRC grant no EP/K032208/1. Finally, we thank the referee for their helpful comments.

\section{Gorenstein ring spectra}
In this section we first review the notions of Gorenstein ring spectrum, relative Gorenstein morphism, and Gorenstein duality, as studied by Dwyer, Greenlees, and Iyengar \cite{dgi_duality}. Then, we prove a result for recognizing when certain morphisms of ring spectra are relative Gorenstein, and use this in the next section to prove an ascent theorem for local Gorenstein duality.

\subsection{The Gorenstein condition}
Suppose $A$ is a (discrete) Noetherian commutative local ring with residue field $k$, then it is a theorem of Serre that $A$ is regular if and only if $k$ has a resolution of finite length by free $A$-modules. For the associated map of Eilenberg--MacLane spectra $HA \to k$, this implies that $k$ is in the thick subcategory of $\Mod_{HA}$ generated by ${HA}$ itself. This leads more generally to the definition of a regular morphism of ring spectra, where we say that a morphism of ring spectra $R \to k$ is regular if $k$ is a compact $R$-module, i.e., $k$ is in the thick subcategory of $\Mod_R$ generated by $R$. However, a weaker notion of regularity is often useful.

\begin{defn}\label{def:proxysmall}
	Let $k$ be a commutative ring spectrum (unless otherwise stated, $k$ is assumed to be a field throughout this paper). A morphism of ring spectra $R \to k$ is called proxy-regular if there exists another $R$-module $K$ called a Koszul complex, such that $K$ is a compact $R$-module, $K$ is in the thick subcategory of $\Mod_R$ generated by $k$, and $k$ is in the localizing subcategory generated by $K$ in $\Mod_R$.  If $K = R$ itself, then we say that $R \to k$ is cosmall. If $K=k$ then we say that $R \to k$ is small.
\end{defn}

Returning to the commutative algebra example, if $A$ is a commutative local Noetherian ring as above, then $HA \to k$ is always proxy-regular, where we can take $K$ to correspond to the usual Koszul complex \cite[Section 5.1]{dgi_duality} associated to a sequence of generators of the maximal ideal of $A$.

We recall that a commutative Noetherian local ring $A$ with residue field $k$ is Gorenstein if and only if $\Ext_A^*(k,A)$ is a one-dimensional $k$-vector space. This leads to the following definition, which is actually a special case of a more general definition due to Dwyer--Greenlees--Iyengar, see \cite[Proposition 8.4]{dgi_duality}.
\begin{defn}
	We say that a morphism $R \to k$ of ring spectra is Gorenstein of shift $a$ if it is proxy-regular and there is an equivalence $\Hom_R(k,R) \simeq \Sigma^ak$ of $k$-modules. More generally, we say that a morphism of ring spectra $S \to R$ is relatively Gorenstein of shift $a$ if $\Hom_S(R,S) \simeq \Sigma^a R$. 
\end{defn}

\subsection{Gorenstein duality}\label{sec:gorduality}
Note that if $R$ admits the structure of a $k$-algebra, and $R \to k$ is Gorenstein, we have an equivalence of $k$-modules
 \begin{equation}\label{eq:orientable}
\Hom_R(k,R) \simeq \Sigma^a k \simeq \Sigma^a\Hom_R(k,\Hom_k(R,k)),
 \end{equation}
 where the last equivalence follows by adjunction.  One important observation of Dwyer--Greenlees--Iyengar is that since $\Hom_R(k,R)$ has an action by $\cE = \Hom_R(k,k)$, if $R \to k$ is Gorenstein, then $\Sigma^a k$ admits the structure of a right $\cE$-module. Likewise, the second equivalence of \eqref{eq:orientable} equips $\Sigma^ak$ with an a priori different right $\cal{E}$-module structure.  

 Restricting to the case of augmented $k$-algebras, this leads to the following definition.\footnote{Without the assumption that $R$ is a $k$-algebra, the definition is more complicated, and involves the notion of Matlis lift of $k$, see \cite[Section 6]{dgi_duality}.}
 
 \begin{defn}\label{def:orientable}
		Let $R$ be an augmented $k$-algebra. If $R$ is Gorenstein, then it is orientable if the two right $\cal{E}$-actions on $\Sigma^ak$ agree, i.e., 
		\[
		\Hom_R(k,R) \simeq \Sigma^a\Hom_R(k,\Hom_k(R,k))
		\] is an equivalence of right $\cal{E}$-modules. 
\end{defn}

One interesting consequence of the Gorenstein condition is the duality that it often implies. To explain this, we introduce some further terminology. If $M$ is an $R$-module, we let $\Cell^R_k(M)$ denote the $k$-cellular approximation of $M$; that is, $\Cell^R_k(M)$ is in the localizing subcategory in $\Mod_R$ generated by $k$, and there is a morphism $\Cell^R_k(M) \to M$ that induces an equivalence on $\Hom_R(k,-)$. If the ring spectrum $R$ is clear, we will usually just write $\Cell_k(M)$. For example, if $A$ is a local Noetherian ring with residue field $k$, then taking $R = HA$ and $M$ a discrete $A$-module, we have that $\pi_*\Cell_k(HM)$ is the local cohomology $H_{\frak m}^*(M)$ of $M$.

If $R \to k$ is proxy-regular, then $k$-cellularization has a particularly simple formula, namely 
\begin{equation}\label{lem:cellproxy}
\Cell_k(M) \simeq \Hom_R(k,M) \otimes_{\cal{E}}k,
\end{equation} where, as previously, $\cal{E} = \Hom_R(k,k)$.
Moreover, $k$-cellular approximation is smashing, that is, $\Cell_k (M) \simeq (\Cell_k(R) )\otimes_R M$	for any $M \in \Mod_R$. Proofs of these claims can be found in \cite[Lemma 6.3]{greenlees_hi} and \cite[Lemma 6.6]{greenlees_hi}. We then have the following, see \cite[Section 18.B]{greenlees_hi}. 

\begin{lem}\label{prop:Gorensteinduality}
		Suppose that $R$ is an augmented $k$-algebra such that $R \to k$ is orientable Gorenstein of shift $a$. Then, there is an equivalence of $R$-modules
		\[
\Cell_k(R) \simeq \Sigma^a \Cell_k(\Hom_k(R,k)).
		\]
		
\end{lem}
\begin{proof}		
		Using the Gorenstein orientable condition there are equivalences of right $\cal{E}$-modules
		\[
		\Hom_R(k,\Cell_k(R)) \simeq \Hom_R(k,R) \simeq \Sigma^a\Hom_R(k,\Hom_k(R,k)).
		\]
		By \eqref{lem:cellproxy}, applying $-\otimes_{\cal{E}}k$ to the latter equivalence gives rise to an equivalence of $R$-modules
		\[
\Cell_k(R) \simeq \Sigma^a\Cell_k(\Hom_k(R,k)).\qedhere
		\]
\end{proof}

%\begin{rem}\label{rem:coconnective}
%If $R$ is additionally coconnective with $\pi_0R \cong k$, then the $R$-module $\Hom_k(R,k)$ is $k$-cellular since $\Hom_k(R,k)$ is bounded below  \cite[Remark 3.17]{dgi_duality}, so that we can deduce
%		\[
%\Cell_k(R) \simeq \Sigma^a\Hom_k(R,k). 
%		\]
%		\end{rem}

\begin{defn}\label{defn:duality}
		Let $R$ be an augmented $k$-algebra. We say that $R$ has Gorenstein duality of shift $a$ if there is an equivalence of $R$-modules
		\[
\Cell_k(R) \simeq \Sigma^a \Hom_k(R,k). 
		\]
	\end{defn}

\begin{prop}\label{lem:orientable-duality}
If $R$ is a coconnective augmented $k$-algebra with $\pi_0R \cong k$ such that  $R \to k$ is Gorenstein orientable, then it satisfies Gorenstein duality. 
\end{prop}
\begin{proof}
 The $R$-module $\Hom_k(R,k)$ is $k$-cellular since $\Hom_k(R,k)$ is bounded below  \cite[Remark 3.17]{dgi_duality}. We can then deduce from \Cref{prop:Gorensteinduality} that $\Cell_k(R) \simeq \Sigma^a\Hom_k(R,k)$.
\end{proof}
Finally, we point out a useful way of recognizing when $R \to k$ is orientable Gorenstein. We need the following definition \cite[Section 8.11]{dgi_duality}.
\begin{defn}
	A $k$-algebra $R$ is said to satisfy Poincar\'e duality of dimension $a$ if there is an equivalence $R\rightarrow \Sigma^a\Hom_k(R,k)$ of $R$-modules.
\end{defn}
\begin{lem}\label{cor:Gorenteincosmall}
Let $R$ be an augmented $k$-algebra which is proxy-regular.  If $R$ satisfies Poincar\'e duality of dimension $a$, then it is orientable Gorenstein of shift $a$. If $R$ is additionally cosmall and coconnective, with $\pi_0R \cong k$, then the reverse implication is true.  
\end{lem}	

\begin{proof}
 By \cite[Proposition 8.12]{dgi_duality}, $R$ is Gorenstein of shift $a$ and we get orientability by applying $\Hom_R(k,-)$ to the equivalence $R\simeq \Sigma^a \Hom_k(R,k)$.

On the other hand, if $R$ is also cosmall and coconnective with $\pi_0R \cong k$, then by \Cref{prop:Gorensteinduality} and \Cref{lem:orientable-duality}, we have that
\[R\simeq \Cell_k(R)\simeq  \Sigma^a\Cell_k(\Hom_k(R,k))\simeq \Sigma^a \Hom_k(R,k),
\]
so that $R$ satisfies Poincar\'e duality of dimension $a$. 
\end{proof}

\begin{rem}\label{rem:gorenstein_not_duality}
Unlike in algebra, if a map of commutative ring spectra $R \to k$ is Gorenstein, then $R$ does not necessarily satisfy Gorenstein duality. For example, let $M$ be a non-oriented compact manifold, e.g., $M = \mathbb{RP}^2$, and take $R = C^*(M;\Z/4)$ the ring spectrum of $\Z/4$-valued cochains on $M$. Then $R$ is Gorenstein, see \cite[Example 11.2(ii)]{Greenlees2007First}, but cannot satisfy Gorenstein duality. 

Indeed, we will argue that for a connected finite $CW$-complex $X$ and a discrete commutative self-injective ring $k$, the ring spectrum $R = C^*(X;k)$ satisfies Gorenstein duality if and only if $X$ satisfies Poincar\'e duality with respect to $k$. Since $M$ is not orientable, it does not satisfy Poincar\'e duality with respect to the self-injective ring $k = \Z/4$, thus $C^*(M;\Z/4)$ cannot satisfy Gorenstein duality. 

In order to prove the claim, we note that, on the one hand, $X$ being a finite $CW$-complex implies that $C^*(X;k) \to k$ is cosmall \cite[Section 5.5(1)]{dgi_duality}, which shows that $\Cell_k(R) \simeq R$. On the other hand, the spectral sequence
\[
\Ext_k^*(\pi_*R,k) \implies \pi_*\Hom_k(R,k)
\]
collapses to an isomorphism $\pi_*\Hom_k(R,k) \cong \Ext^0_k(\pi_*R,k)$. Therefore, there is an equivalence $\Cell_k(R) \simeq \Sigma^a \Hom_k(R,k)$ if and only if
\[
H^{*}(X;k) \cong \pi_{-*}(R) \cong \pi_{-*}\Cell_k(R) \cong \pi_{-*}\Sigma^a \Hom_k(R,k) \cong \Ext^0_k(H^{*-a}(X;k),k),
\]
i.e., if and only if $X$ satisfies Poincar\'e duality with coefficients in $k$. 
\end{rem}

\subsection{Local Gorenstein duality}\label{sec:localgorenstein}

	Classically, if $A$ is a discrete commutative Gorenstein ring, then so is the localization $A_{\fp}$ for any prime ideal $\fp \in \Spec(A)$. The proof involves the characterization of Gorenstein rings as those rings with finite injective dimension, and so there is no obvious generalization to the case of ring spectra. Rather, we will identify conditions where the duality condition of \Cref{defn:duality} localizes. For this we need to explain what we mean by localizing a ring spectrum $R$ at a prime ideal $\fp \in \Spec^h(\pi_*R)$. Namely, following \cite{benson_local_cohom_2008} or \cite{bhv2} we explain how, given any $\frak p \in \Spec^h(\pi_*R)$, we can define a functor $\Gamma_{\frak p} \colon \Mod_R \to \Mod_R$, which is a spectral version of classical local cohomology.

We briefly describe one way to construct $\Gamma_{\frak p}$. For any such $\frak p$ there exists a ring spectrum ${R}_{\frak p}$ with homotopy $(\pi_*R)_{\frak p}$, and a natural morphism $R \to R_{\frak p}$, see \cite[Chapter V.1]{ekmm} for example. By extension of scalars there is a functor $\Mod_R \to \Mod_{R_{\frak p}}$ sending $M$ to $M_{\frak p} = M \otimes_R R_{\frak p}$. We can then construct a Koszul object $\kos{R_{\frak p}}{\frak p}$ inside $\Mod_{R_{\frak p}}$, see \cite[Section 3.1]{bhv2}. The localizing subcategory generated by $\kos{R_{\frak p}}{\frak p}$ inside of $\Mod_{R_{\frak p}}$ is denoted $\Mod_{R_{\frak p}}^{\frak{p}-\mathrm{tors}}$, the category of $\frak p$-local and $\frak p$-torsion objects. The inclusion $\Mod_{R_{\frak p}}^{\frak{p}-\mathrm{tors}} \to \Mod_{R_{\frak p}}$ has a right adjoint $\Gamma_{\cal{V}(\frak p)}$ (this is equivalent to the local cohomology functor constructed by Greenlees and May in \cite[Section 3]{greenleesmay_completions}). We then define $\Gamma_{\frak p}M = \Gamma_{\cal{V}(\frak p)}M_{\frak p}$. 

Under some conditions we can identify $\Cell_k$ with a local cohomology functor. 
\begin{lem}\label{lem:cellequiv}
	Let  $k$ be a field, and $R$  a coconnective commutative augmented $k$-algebra. Assume that $\pi_*R$ is a Noetherian local ring and that the augmentation map induces an isomorphism $\pi_0R \cong k \cong (\pi_*R)/\frak m$. Then, the functors $\Cell_k$ and $\Gamma_{\frak m}$ are equivalent. 
\end{lem}
\begin{proof}
	This is a consequence of the proof of \cite[Proposition 9.3]{dgi_duality}. We sketch the details for the benefit of the reader. First note that $\Gamma_{\frak m} = \Gamma_{\cal{V}(\frak m)}$. It then suffices to show that $\Gamma_{\cal{V}(\frak m)}M \simeq \Cell_k(M)$.

	Suppose now that $\frak m = (x_1,\dots,x_n)$. Let $K_{\infty} = K_{\infty}(x_1) \otimes_R \cdots \otimes_R K_{\infty}(x_n)$, where $K_{\infty}(x_i)$ is the fiber of the map $R \to R[1/x_i]$. 
	 By \cite{greenleesmay_completions} there is an equivalence $\Gamma_{\cal{V}(\frak m)}M = K_{\infty} \otimes_R M$.  By the proof of \cite[Proposition 9.3]{dgi_duality} the assumptions of the lemma give rise to an equivalence $K_\infty \otimes_R M \simeq \Cell_k(M)$ and we are done. 
\end{proof}
Thus, for our local version of Gorenstein duality we will replace $\Cell_k(R)$ with $\Gamma_{\fp}R$.  The next question is what the analog of $\Hom_k(R,k)$ should be in general. From now on we write $\cal{I}_R = \Hom_k(R,k)$. Suppose that we are still under the conditions of  \Cref{lem:cellequiv}, then $\pi_* \cal{I}_R \cong \pi_*\Hom_k(R,k) \cong \Hom_k(\pi_*R,k) \cong I_{\frak m}$, the injective hull of $(\pi_*R)/\frak m \cong k$, see \cite[Example 3.41]{lambook}. By Brown representability, there is an $R$-module spectrum $T_R(I_{\frak m})$ such that $\pi_*T_R(I_{\frak m}) \cong I_{\frak m}$. Then $\cal{I}_R\simeq T_R(I_{\frak m})$. More generally, for any injective $\pi_*R$-module $I$, as in \cite[Section 4]{bhv2} we can construct an $R$-module $T_R(I)$ such that $\pi_*T_R(I) \cong I$. These spectra are characterized by the property that for any $M \in \Mod_R$ there is an isomorphism of graded $\pi_*R$-modules
\[
\pi_*\Hom_R(M,T_R(I))\cong \Hom_{\pi_*R}(\pi_*M,I). 
\]
We let $I_{\frak p}$ denote the injective hull of $(\pi_*R)/\frak p$.  The spectra $T_R(I_{\frak p})$ are our local substitutes for $\cal{I}_R$. Together, we get the following definition. 

\begin{defn}\label{defn:localizedgorensteinduality}
		Let $R$ be a ring spectrum. We say that $R$ satisfies local Gorenstein duality with shift $a$ if, for each $\frak p \in \Spec^h(\pi_*R)$ of dimension $d$,\footnote{The dimension (also known as the coheight) of a prime ideal $\frak p$ is the Krull dimension of $\pi_*R/\frak p$.} there is an equivalence $\Gamma_{\frak p}R \simeq \Sigma^{a+d} T_R(I_{\frak p})$. 
\end{defn}

\begin{rem}
	As we have discussed, under the conditions of \Cref{lem:cellequiv}, this reduces to the Gorenstein duality condition 
	\[
\Cell_k (R) \simeq \Sigma^a \Hom_k(R,k)
	\]
	in the case that $\frak p$ is the maximal ideal $\frak m$. 
\end{rem}
Finally, we point out some properties of rings satisfying local Gorenstein duality. Here we denote the internal shift functor in graded modules by $\Sigma$ as well. 
\begin{thm}\label{thm:green_lyub}
	Let $R$ be a ring spectrum satisfying local Gorenstein duality of shift $a$. Then the following hold.
\begin{enumerate}
	\item There is an isomorphism of $R$-modules $\pi_*\Gamma_{\fp}R \cong \Sigma^{a+d}I_{\frak p}$.
	\item There is a spectral sequence
	\[
E_2^{s,t} \cong H_{\frak p}^{-s,t}(\pi_*R)_{\frak p} \implies \pi_{s+t-a-d}T_R(I_{\frak p}) \cong \Sigma^{a+d-s-t}I_{\frak p}. 
	\]
	\item $\pi_*R$ is generically Gorenstein, i.e., the localization at any minimal prime ideal is Gorenstein. 
	\item There are no nontrivial $R$-module phantom maps into $\Gamma_{\frak p}R$. 
\end{enumerate}
\end{thm}
\begin{proof}
	(1) is an immediate consequence from the definition of local Gorenstein duality and the fact that $\pi_*T_R(I_{\frak p}) \cong I_{\frak p}$. (2) follows from (1) and the local cohomology spectral sequence, see for example \cite[Proposition 3.19]{bhv2}. (3) follows from the spectral sequence in (2). Indeed, if $\frak p$ is minimal, then the localized ring $(\pi_*R)_{\fp}$ is of dimension 0, and hence $H_{\frak p}^{-s,t}(\pi_*R)_{\frak p} = 0$ whenever $s \ne 0$, and the spectral sequence collapses. 
	 For (4), it is shown in \cite[Lemma 3.2]{BarthelHeardValenzuela2018algebraic} that there are no nontrivial phantom maps into $T_R(I_{\frak p})$, and hence no phantom maps into $\Gamma_{\frak p}R$. 
\end{proof}
\begin{rem}\label{rem:rickard}
  The functors $\Gamma_{\frak p}$ can be defined more generally for any triangulated category with a central action of a (discrete) commutative Noetherian ring $A$. In the case of the stable module category of a finite group $G$ over a field $k$, the ring $A = H^*(BG;k)$ is the group cohomology, and $\Gamma_{\frak p}k$ corresponds to the Rickard idempotent denoted $\kappa_V$ in \cite{BensonCarlsonRickard1997Thick} for $V$ an irreducible subvariety of $\Spec^h(H^*(BG;k))$ corresponding to $\frak p$. For a detailed discussion of the comparison, see page 30 of \cite{benson_local_cohom_2008}. As shown in \cite[Theorem 2.4]{bg_localduality}, the Tate cohomology $\widehat H^*(G;\kappa_V) \cong \Sigma^d I_{\frak p}$, the injective hull of $H^*(BG;k)/\frak p$ in the category of $H^*(BG;k)$-modules, shifted by the dimension of $\frak p$. The result \Cref{thm:green_lyub}(1) above is the spectral generalization of this, and shows how the homotopy of the spectrum $\Gamma_{\frak p}R$ is completely determined by the homotopy groups $\pi_*R$. 
\end{rem}
Finally we introduce a useful way of identifying spectra satisfying local Gorenstein duality. The following notion was introduced in \cite[Definition 4.5]{bhv2}.

\begin{defn}\label{defn:alg_gorenstein}
A ring spectrum $R$ with $\pi_*(R)$ local Noetherian of dimension $n$ is algebraically Gorenstein of shift $a$ if $\pi_*(R)$ is a graded Gorenstein ring of the same shift. 
\end{defn}

\begin{prop}\cite[Proposition 4.7]{bhv2}\label{prop:4.7bhv2}
Let $R$ be a ring spectrum. If $R$ is algebraically Gorenstein of shift $a$, then $R$ satisfies local Gorenstein duality of shift $a$.
\end{prop}

\subsection{The relative Gorenstein condition}\label{ssec:relativegorenstein}
Let $R\to k$ and $S\to k$ be morphisms of ring spectra and $S \xr{f} R$ be a morphism of ring spectra over $k$.  We have restriction of scalars $f^*\colon \Mod_R\to \Mod_S$ with left adjoint $f_*$ and right adjoint $f_!$. Note that if $f$ is relatively Gorenstein, then $f_!(S) \simeq \Sigma^aR$. %If $S$ and $Q = R \otimes_S k$ are Gorenstein, then \Cref{thm:gorenstein_ascent} gives a condition to identify when $R$ is Gorenstein. 
The purpose of this section is to identify conditions guaranteeing that $f \colon S \to R$ is relatively Gorenstein. First we observe the following:

\begin{lem}\label{lem:gorenstein_iff}
Let $R\to k$ and $S\to k$ be proxy-regular morphisms of ring spectra, and $f\colon S\to R$ be a relative Gorenstein ring morphism over $k$. Then $S$ is Gorenstein if and only if $R$ is so. 
\end{lem}

\begin{proof} This follows from the following identities where the third one holds because $f$ is relative Gorenstein:
\[\Hom_S(k,S)\simeq \Hom_S(f^*(k),S)\simeq \Hom_R(k, f_!(S))\simeq \Hom_R(k, \Sigma^a R)\simeq \Sigma^a \Hom_R(k, R).\qedhere\]
\end{proof}

At this point, we need to recall the notion of dc-completeness. There is a canonical morphism $R \to \End_{\cal{E}}(k)= \widehat{R}$ induced from the $R$-module action on $k$, and we say that $R$ is dc-complete if this map is an equivalence.  Recall that  we write $\cal{I}_R = \Hom_k(R,k)$. We want to describe hypothesis on $R$ and $S$ which allow us to identify when a morphism is relative Gorenstein.

\begin{prop}\label{prop:completion}
Suppose that $R$ is an augmented $k$-algebra, and that $R \to k$ is orientable Gorenstein of shift $a$, then 
	\[
\Hom_R(\Cell_k(R),\cal{I}_R) \simeq \Sigma^{-a}\widehat{R}. 
	\] 
\end{prop}
\begin{proof}
	By \eqref{lem:cellproxy} and the orientable Gorenstein condition, there is an equivalence \[
	\Cell_k(R) \simeq \Hom_R(k,R) \otimes_{\cal{E}}k \simeq \Sigma^a k \otimes_{\cal{E}}k.\] By the definition of $\cal{I}_R$,
	\[
\Hom_R(\Cell_k(R),\cal{I}_R) \simeq \Hom_k(\Cell_k(R),k). 
	\]
	Substituting for $\Cell_k(R)$, we see that 
	\[
	\begin{split}
\Hom_R(\Cell_k(R),\cal{I}_R) \simeq \Hom_k(\Sigma^a k \otimes_{\cal{E}}k,k)&\simeq \Sigma^{-a}\Hom_\cal{E}(k,k) =\Sigma^{-a} \widehat{R},		
	\end{split}
	\]
	as required.
\end{proof}
We require two more lemmas. The first follows by a simple adjunction argument.
\begin{lem}\label{lem:coind}
	Let $R\to k$ and $S\to k$ be ring homomorphisms and $S \xr{f} R$ be a morphism of ring spectra over $k$. There is an equivalence of $S$-modules $\Hom_S(R,\cal{I}_S) \simeq \cal{I}_R$. 
\end{lem}

For the second lemma, we observe that $k$ is both naturally an $S$-module and an $R$-module. The following compares the $k$-cellularization functor in the two categories.

\begin{lem}\label{lem:basechange}
	Let $S \xr{f} R$ be a morphism of ring spectra over $k$ and  $Q = R \otimes_S k$. Assume that $S \to k$ is proxy-regular with Koszul object $K(S)$.  Consider the following conditions:
\begin{enumerate}
	\item $Q \to k$ is cosmall. 
	\item $R \to k$ is proxy-regular with Koszul object $R \otimes_{S} K(S)$. 
	\item There is an equivalence of $S$-modules $\Cell_k^S(R) \simeq \Cell_k^R(R)$.
\end{enumerate}
	 Then $(1) \implies (2) \implies (3)$.
\end{lem}
\begin{proof}
That $(1) \implies (2)$ is shown in the proof of \cite[Proposition 4.18(1)]{dgi_duality}.

Assuming (2) then, we have the equivalences $\Cell_k^R(R) \simeq \Cell_{K(R)}^R(R) \simeq \Cell_{R \otimes_S K(S)}^R(R)$ and $\Cell_k^S(R) \simeq \Cell_{K(S)}^S(R)$. To show $(2) \implies (3)$ it thus suffices to prove that $\Cell^R_{R\otimes_SK(S)}(R) \simeq \Cell_{K(S)}^S(R)$ as $S$-modules. According to \cite[Lemma 3.1(2)]{shamir_cellular} this follows if $R \otimes_S K(S)$ is $K(S)$-cellular as an $S$-module. Since the category of $S$-modules is generated by $S$, we deduce that $R \otimes_S K(S) \in \Loc_{\Mod_S}(K(S))$, i.e., $R \otimes_S K(S)$ is $K(S)$-cellular as an $S$-module, as needed. 
\end{proof}
We now obtain our result for deducing that a morphism $f\colon S \to R$ is relatively Gorenstein. For this, we say that a morphism $f \colon R \to S$ is finite if $S$ is compact as an $R$-module. 

\begin{prop}\label{thm:localdualityfinite}
	Suppose that we have a finite morphism of augmented $k$-algebras $f \colon S \to R$ over $k$, and let $Q = R \otimes_S k$. Assume that the following conditions are satisfied:
	\begin{enumerate}
	 	\item $R \to k$ and $S \to k$ are orientable Gorenstein of shift $r$ and $s$ respectively.
	 \item $Q \to k$ is cosmall.
	 \item $R$ and $S$ are dc-complete.
	 \end{enumerate}  
	 Then, $f$ is relatively Gorenstein of shift $s-r$. 
\end{prop}
\begin{proof}
		The proof follows \cite[Theorem 7.3(m)]{bg_localduality}, which is actually a special case of this theorem. We have a chain of equivalences 
	\begin{align*}
	\Hom_S(R,S) &\simeq \Sigma^{s}\Hom_S(R,\Hom_S(\Cell^S_k(S),\cal{I}_S))	& [\mathrm{\Cref{prop:completion}} \text{ and } (3)] \\
		& \simeq \Sigma^{s}\Hom_S((\Cell^S_k(S)) \otimes_SR,\cal{I}_S) &\\
		& \simeq \Sigma^{s}\Hom_S(\Cell^S_k(R),\cal{I}_S) & [\mathrm{\Cref{lem:cellproxy}}]\\
		& \simeq \Sigma^{s}\Hom_S(\Cell^R_k(R),\cal{I}_S) & [\mathrm{\Cref{lem:basechange}}]\\
		& \simeq \Sigma^{s}\Hom_R(\Cell_k^R(R),\Hom_{S}(R,\cal{I}_S)) & \\
		& \simeq \Sigma^{s}\Hom_R(\Cell^R_k(R),\cal{I}_R) & [\mathrm{\Cref{lem:coind}}]\\
		& \simeq \Sigma^{s-r}R. & [\mathrm{\Cref{prop:completion}}\text{ and } (3)] 
	\end{align*}
	This is exactly the claim that $\Hom_S(R,S)$ is relatively Gorenstein of shift $s-r$. 
\end{proof}

\section{Gorenstein ascent}

In this section we describe ascent techniques which allow to construct new examples of ring spectra satisfying (local) Gorenstein duality from known examples.

\subsection{Ascent for Gorenstein rings}\label{sec:gorenstein_ascent}

Let $R\to k$ and $S\to k$ be morphisms of ring spectra. Suppose we are given a morphism $f \colon S \to R$ over $k$ with $f$ relatively Gorenstein, then $S$ is Gorenstein if and only if $R$ is Gorenstein, see \Cref{lem:gorenstein_iff}. Now let $Q = R \otimes_S k$. We consider the situation where two out of $R,S$ and $Q$ satisfy Gorenstein duality. Part (1) of the following was already shown in \cite[Section 8.6]{dgi_duality} or \cite[Lemma 19.3]{greenlees_hi}.

\begin{thm}\label{thm:gorenstein_ascent}
	Let $S \xr{f} R $ be a morphism over $k$, and let $Q = R \otimes_S k$. Suppose that the natural morphism $\nu \colon \Hom_S(k,S) \otimes_S R \to \Hom_S(k,R)$ is an equivalence of $S$-modules, and that one of the following conditions is satisfied:
	\begin{enumerate}[(i)]
		\item $S \to k$ is proxy-regular and $Q \to k$ is cosmall.
		\item $S \to k$ is small and $Q \to k$ is proxy-regular.
	\end{enumerate}
Then the following hold. 
\begin{enumerate}
\item If $Q \to k$ and $S \to k$ are Gorenstein of shift $q$ and $s$ respectively, then $R \to k$ is Gorenstein of shift $s + q$.
\item If $S \to k$ and $R \to k$ are Gorenstein of shift $s$ and $r$ respectively, then $Q \to k$ is Gorenstein of shift $r-s$.
\end{enumerate}
\end{thm}
\begin{proof}
	Either of the conditions implies that $R \to k$ is proxy-regular, see \cite[Proposition 4.18]{dgi_duality}, so that $R \to k$ is Gorenstein if and only if $\Hom_R(k,R) \simeq \Sigma^ak$. But by \cite[Proposition 8.6]{dgi_duality} or \cite[Lemma 19.3]{greenlees_hi} the assumption on $\nu$ implies that there is an equivalence of $k$-modules
	\[
\Hom_R(k,R) \simeq \Hom_Q(k,\Hom_S(k,S) \otimes_kQ).
	\]
	The result follows. 
\end{proof}
\begin{rem}\label{rem:nu}
	The natural map $\nu$ is an equivalence if (but not only if) either $R$ or $k$ are small as $S$-modules.
\end{rem}

\subsection{Ascent for local Gorenstein duality}
In \Cref{sec:gorenstein_ascent} we recalled the Gorenstein ascent theorem of Dwyer--Greenlees--Iyengar, namely that for a finite morphism $S \xr{f} R$ over $k$ if $S$ and $Q = R \otimes_S k$ are Gorenstein, then so is $R$. In \cite[Theorem 4.27]{bhv2}, which we recall now, we gave a criterion for ascent of local Gorenstein duality. Note that in the following we do not need to assume that $R$ and $S$ are $k$-algebras. 
\begin{prop}\label{prop:bhv_ascent}
Let $R$ and $S$ be ring spectra. Suppose that $S$ satisfies local Gorenstein duality of shift $s$ and that $f \colon S \to R$ is a finite morphism and is relatively Gorenstein of shift $r-s$, then $R$ satisfies local Gorenstein duality of shift $r$. 
\end{prop}

Combined with \Cref{thm:localdualityfinite} we deduce the following. 

\begin{prop}\label{cor:localdualityfinite}
	Let $S \xr{f} R$ be a finite morphism of augmented $k$-algebras and $Q = R \otimes_S k$. Assume that the following conditions are satisfied:
	\begin{enumerate}
	 	\item $R \to k$ and $S \to k$ are orientable Gorenstein of shift $r$ and $s$ respectively.
	 \item $Q \to k$ is cosmall.
	 \item $R$ and $S$ are dc-complete.
	 \end{enumerate}  
	 Then, if $S$ satisfies local Gorenstein duality of shift $s$, then $R$ satisfies local Gorenstein duality of shift $r$. 
\end{prop}

\subsection{Cochain algebras}

In this section we specialize the results of the previous subsection to ring spectra obtained as cochains on spaces. We recall that we write $C^*(X;k)$ for the ring spectrum of $k$-valued cochains on $X$ defined as the function spectrum $\Hom_{\Sp}(\Sigma_+^{\infty}X,Hk)$. In particular, there is an isomorphism $\pi_*C^*(X;k) \cong H^{-*}(X;k)$. 

%\subsection{The Gorenstein condition for cochain algebras}
%We now specialize to the case where the ring spectra are of the form $C^*(X;k)$ for a space $X$. 
It is important to have in mind that the object we are interested in is $R=C^*(X;k)$, not $X$ itself, which means that we can have different spaces giving rise to the same ring spectrum $R$. For example, $R=C^*(B\mathbb Z/p; \F_q) \simeq C^*(*;\F_q)$ if $p$ and $q$ are coprime.

If $k$ is a field and $X$ is a space, then we denote the Bousfield $k$-completion of $X$ by $X^{\wedge}_k$. If $X$ is $k$-good, then $C^*(X;k)\simeq C^*(X^\wedge_k;k)$ and in this case we can assume that $X$ is $k$-complete. For example, if $\pi_1X$ is finite, then $X$ is $k$-good for $k=\mathbb Q$ and $k=\F_p$ for any prime $p$, and therefore $X^\wedge_k$ is $k$-complete (see \cite[I.5.2, VII.3.2, VII.5.1]{bousfield_homotopy_1972}).

Given a space $X$, the ring spectrum of cochains will be well behaved when it satisfies certain hypothesis which we will assume mostly through the rest of the paper:
\begin{defn}\cite[Section 4.22]{dgi_duality}\label{def:emtype}
A space $X$ is said to be of Eilenberg--Moore type (EM-type) if $X$ is connected, $H^*(X;k)$ is of finite type, and
	\begin{enumerate}
		\item $X$ is simply connected when $k=\Q$, or
		\item $k$ is of characteristic $p$ and $\pi_1X$ is a finite $p$-group. 
	\end{enumerate}
\end{defn}

\begin{rem}
A space of EM-type is $k$-good and so we can assume always that $X$ is $k$-complete when considering its ring spectrum of cochains with coefficients in $k$.
\end{rem}

We are interested in these properties for the following reason. Suppose we are given a homotopy pullback square of spaces
\[
\begin{tikzcd}
	Y \times_X Z \arrow{r} \arrow{d} & Z \arrow{d} \\
	Y \arrow{r} & X.
\end{tikzcd}
\]
If $X$ is of EM-type, then the homotopy pullback gives rise to an equivalence \cite[Corollary 1.1.10]{dagxiii}
\[
C^*(Y \times_X Z;k) \simeq C^*(Y;k) \otimes_{C^*(X;k)} C^*(Z;k). 
\]
In particular, we obtain:
\begin{lem}\label{lem:emss}
	Let $F \to Y \to X$ be a fiber sequence of spaces where $X$ is of EM-type. Then
	\[
C^*(F;k) \simeq C^*(Y;k) \otimes_{C^*(X;k)} k.
	\]
\end{lem}
Under some hypothesis on $X$, a morphism $C^*(X;k) \to k$ that is Gorenstein is also automatically orientable.
	\begin{lem}\label{lem:orientable}
		Let $X$ be a connected space such that $H^*(X;k)$ is of finite type. Suppose that:
		\begin{enumerate}
		 	\item $X$ is simply connected with $k = \Q$, or
      \item $k$ is field of characteristic $p$, and $\pi_1X$ is a finite $p$-group. 
		 	\item $k$ is a finite field of characteristic $p$ and $\pi_1X$ is a pro-$p$ group. 
		 \end{enumerate}
		 Then, if $C^*(X;k) \to k$ is Gorenstein, it is orientable. In particular, the conditions of the lemma hold if $X$ is a space of EM-type.
	\end{lem}
\begin{proof}
	We first check that  $\cal{E}=\Hom_R(k,k)\simeq C_*(\Omega X;k)$. If $X$ is simply connected with $k = \Q$ then it follows from the strong convergence of the Eilenberg--Moore spectral sequence. Otherwise,  the action of $\pi_1X$ on $H^*(X;k)$ is nilpotent: if $k$ is a finite field of characteristic $p$, the action factors then through a finite quotient which is a $p$-group since $H^*(X;k)$ is of finite type. Then again the  strong convergence of the Eilenberg--Moore spectral sequence shows that $\cal{E}=\Hom_R(k,k)\simeq C_*(\Omega X;k)$. 
	
	We show that $\cal{E}$ has a unique action on $k$. This action factors through $\pi_0\cal{E}\cong \pi_1X$ since $k$ is an Eilenberg--MacLane spectrum.  The case where $k$ is of characteristic $p$ and $\pi_1X$ is a finite $p$-group is \cite[Lemma 18.2]{greenlees_hi}, and follows because $\cal{E}$ is a $k$-algebra, and acts through $\pi_0(\cal{E}) \cong H_0(\Omega X) \cong k[\pi_1X]$. Because $\pi_1X$ is a finite $p$-group, and $k$ has characteristic $p$, this action must be unique. 

   If, $\pi_1X$ is a pro-$p$ group and $k$ is finite, then the action map factors through a finite quotient of $\pi_1X$, which is a finite $p$-group, and hence the result also follows in this case. In the rational case, since $X$ is simply connected, the same argument as \cite[Lemma 18.2]{greenlees_hi} shows that $k$ has a unique $\cal{E}$-module structure, and hence is orientably Gorenstein. 
\end{proof}

In the previous subsection we identified conditions to ascend local Gorenstein duality along a finite morphism. In light of \Cref{rem:nu}, we are interested in conditions which ensure that the induced morphism on cochains for a map of spaces $f \colon Y \to X$ is finite. To this end, we have the following, due to Greenlees--Hess--Shamir \cite{GreenleesHessShamir2013Complete} in the rational case, and Benson--Greenlees--Shamir \cite{shamir_pcochains} in the characteristic $p$ case. 

\begin{lem}\label{lem:spacesfinit}
	Let $f \colon Y \to X$ be a map of spaces with homotopy fiber $F$ and $H^*(F;k)$ finite dimensional. If $k = \Q$, then $f \colon C^*(X;k) \to C^*(Y;k)$ is finite. If $k$ has characteristic $p$, then the same conclusion holds if additionally $X$ and $Y$ are $p$-complete spaces with fundamental groups finite $p$-groups. 
\end{lem}
\begin{proof}
	The rational case is proved in \cite[Lemma 4.7]{GreenleesHessShamir2013Complete}, while the characteristic $p$ case is \cite[Lemma 3.4]{shamir_pcochains}. 
\end{proof}
	
We now present a cochain version of \Cref{thm:gorenstein_ascent}.
	
\begin{thm}\label{thm:ascent_spaces}
	Suppose that $g \colon Y \to X$ is a morphism of spaces of EM-type ($p$-complete if the characteristic of $k$ is $p$) with fiber $F$, such that $H^*(F;k)$ is finite-dimensional, and that $C^*(X;k) \to k$ is proxy-regular. Then the following hold: 
\begin{enumerate}
\item If $C^*(F;k) \to k$ and $C^*(X;k) \to k$ are Gorenstein of shift $q$ and $s$ respectively, then $C^*(Y;k) \to k$ is Gorenstein of shift $s + q$. 
\item If $C^*(Y;k) \to k$ and $C^*(X;k) \to k$ are Gorenstein of shift $r$ and $s$ respectively, then $C^*(F;k) \to k$ is Gorenstein of shift $r-s$.
\end{enumerate}
\end{thm}
\begin{proof}
	By \Cref{lem:emss,lem:spacesfinit} there is a finite morphism $C^*(X;k) \xr{f} C^*(Y;k)$ and $C^*(F;k) \simeq C^*(Y;k) \otimes_{C^*(X;k)}k$. Moreover, $H^*(F)$ finite-dimensional implies that $C^*(F;k) \to k$ is cosmall \cite[Section 5.5(2)]{dgi_duality}. We are thus in the situation of \Cref{thm:gorenstein_ascent}(i).
\end{proof}

\begin{rem}
It worth pointing out that $C^*(Y;k)$ is Gorenstein if, for example, $C^*(F;k)$ is a Poincar\'e duality algebra by \Cref{thm:ascent_spaces} and \Cref{cor:Gorenteincosmall}.
\end{rem}

We now specialize the results on relative Gorenstein duality and local Gorenstein duality to cochain algebras. 
\begin{thm}\label{thm:relgorensteinspaces}
		Let $g \colon Y \to X$ be a morphism of spaces of EM-type ($p$-complete if the characteristic of $k$ is $p$) with fiber $F$.  Suppose that $H^*(F;k)$ is finite-dimensional, and that $C^*(X;k)$ is Gorenstein of shift $s$.
\begin{enumerate}
	\item If $C^*(Y;k)$ is Gorenstein of shift $r$, then $g^* \colon C^*(X;k) \to C^*(Y;k)$ is relatively Gorenstein of shift $s-r$.
	\item If, in addition, $C^*(X;k)$ satisfies local Gorenstein duality of shift $s$, then $C^*(Y;k)$ satisfies local Gorenstein duality of shift $r$. 
\end{enumerate}
\end{thm}
\begin{proof}
	The Eilenberg--Moore condition shows that $C^*(F;k) \simeq C^*(Y;k) \otimes_{C^*(X;k)}k$. We therefore must verify the three conditions given in \Cref{thm:localdualityfinite} and \Cref{cor:localdualityfinite}, applied to morphisms $C^*(X;k) \xr{f} C^*(Y;k)$ and $C^*(F;k) \to k$.  We note that by assumption on $H^*(F;k)$, the morphism $f$ is finite by \Cref{lem:spacesfinit}. 
\begin{enumerate}
	\item  By assumption $C^*(X;k)$ is Gorenstein of shift $s$. Further, $C^*(Y;k)$ is Gorenstein of shift $r$, either by assumption, or by \Cref{thm:ascent_spaces} in the case that $C^*(F;k)$ is a Poincar\'e duality algebra. Since $X$ and $Y$ are assumed to be of EM-type, they are automatically orientable Gorenstein by \Cref{lem:orientable}.
	\item The assumption that $H^*(F;k)$ is finite-dimensional implies that $C^*(F;k) \to k$ is cosmall \cite[Section 5.5(2)]{dgi_duality}.
	\item Since $X$ and $Y$ are assumed to be of EM-type, they are automatically dc-complete \cite[Section 7.B]{greenlees_hi}. 
\end{enumerate}
Now we apply \Cref{thm:localdualityfinite} to see that $f$ is relatively Gorenstein of shift $s-r$, as claimed. If, in addition, $C^*(X;k)$ satisfies local Gorenstein duality of shift $s$, then \Cref{cor:localdualityfinite} implies that $C^*(Y;k)$ satisfies local Gorenstein duality of shift $r$. 
\end{proof}

\section{Examples}

This section is devoted to relevant examples of type $C^*(X;k)$ coming from $H$-spaces, finite loop spaces, Lie groups, and Noetherian rational spaces.

\subsection{Spaces with Gorenstein cohomology ring}
A first source of examples is given by algebraically Gorenstein ring spectra $R$. When $R=C^*(X;k)$, this means that $H^*(X;k)$ is a Gorenstein ring, see \Cref{defn:alg_gorenstein}. As recalled in \Cref{prop:4.7bhv2}, algebraically Gorenstein implies local Gorenstein duality, and so, by virtue of \Cref{lem:cellequiv}, Gorenstein duality holds. Secondly, if $R$ is proxy-regular then we also deduce that $C^*(X;k)\rightarrow k$ is Gorenstein by applying $\Hom_R(k,-)$ to the Gorenstein duality, see the proof of \Cref{cor:Gorenteincosmall}.  

Examples of Gorenstein rings are given by finite duality algebras (which are Gorenstein of Krull dimension zero) and polynomial algebras.

\begin{ex}
Let $M$ be an orientable compact manifold, then $M$ is homotopy equivalent to a finite $CW$-complex, so $C^*(M;k) \to k$ is cosmall \cite[Section 5.5]{dgi_duality} and hence proxy-regular. Therefore $C^*(M;k)\rightarrow k$ is Gorenstein and satisfies Gorenstein local duality. 
\end{ex}

Other examples come from the theory of $H$-spaces. 

\begin{ex}\label{ex:hspace}
If $X$ is an $H$-space with finite cohomology $H^*(X;k)$, then by the classification  \cite{hopf_uber_1941,mm65} of finite-dimensional Hopf algebras over a perfect field $k$, $H^*(X;k)$ is a Poincar\'e duality algebra. In particular this includes mod $p$ finite loop spaces when $k=\F_p$, that is, loop spaces with finite mod $p$ cohomology. 
\end{ex}

Another source of examples is given by spaces with polynomial algebra. If $X$ is $k$-good with $H^*(X;k)$ polynomial then it is also algebraically Gorenstein. 

\begin{ex}\label{ex:U(n)}
The group cohomology of $U(n)$ is given by $H^*(BU(n);k) \cong k[c_1,\ldots,c_n]$, where $c_i$ is the $i$-th Chern class, with degree $2i$, and $k=\F_p,\Q$. In particular, the ring is regular, hence Gorenstein, and therefore satisfies local Gorenstein duality  of shift $\dim(U(n))$, see \Cref{prop:4.7bhv2}.
\end{ex}

A less obvious situation is when $X$ is a connected $H$-space with Noetherian mod $p$ cohomology, which combines the two previous basic examples.

\begin{prop}
Let $X$ be a connected $H$-space with Noetherian cohomology, then $H^*(X;\F_p)$ is Gorenstein. In particular, $C^*(X;\F_p)$ is algebraically Gorenstein, and therefore local Gorenstein duality holds.
\end{prop}
\begin{proof}
The mod $2$ cohomology is given by
\[
H^*(X;\F_2) \cong \F_2[x_1,\ldots,x_r] \otimes \F_2[y_1,\ldots,y_s]/(y_1^{2^{a_1}},\ldots,y_s^{2^{a_s}}) \]
see, e.g., \cite[Equation (5)]{BrotoCrespo1999spaces}, while for $p$ odd we have
\[
H^*(X;\F_p) \cong \frac{\F_p[y_1,\ldots,y_s]}{(y_1^2,\ldots,y_s^2)} \otimes \F_p[\beta y_1,\ldots, \beta y_k,x_{k+1},\ldots,x_n] \otimes \frac{\F_p[z_1,\ldots,z_t]}{(z_1^{p^{a_1}},\ldots,z_t^{p^{a_t}})},
\]
where $\beta$ is the Bockstein, see \cite[Corollary 2.7]{Crespo}. These are clearly Cohen--Macaulay rings, and are Gorenstein because the quotient by $(x_1,\ldots,x_r)$ for $p =2$ or $(\beta y_1,\ldots,\beta y_k,x_{k+1},\ldots,x_n)$ for $p$ odd is a Poincar\'e duality algebra, see Proposition I.1.4 and the remark on the same page of \cite{MeyerSmith2005Poincare}. 
%A description of the mod $p$ cohomology can be found in : $A=H^*(X;\F_p)\cong P\otimes D$ where $P$ is a polynomial algebra and $D$ is a finite Hopf algebra, in particular, a finite Poincar\'e duality algebra by \Cref{ex:hspace}. Since $A$ is a Noetherian local ring, $A$ is Gorenstein if and only if $A/(x)$ is as well, where $x$ is a nonzero divisor in the maximal ideal \cite[Proposition 3.1.19]{cmrings}. Thus, it follows that $A$ is Gorenstein if and only if $D$ is so. \todo{graded version of this? }
\end{proof}

\subsection{Classifying spaces of $p$-compact groups}
We consider the $p$-compact groups of Dwyer--Wilkerson \cite{dwyerwilkerson_finloop}. As such, we fix $k = \F_p$, and for a space $X$ we write $H^*(X)$ for $H^*(X;\F_p)$ and similarly for $C^*(BX)$. 

We recall that a $p$-compact group is a loop space $X \simeq \Omega BX$ such that $X$ is $\F_p$-finite and $BX$ is a pointed connected $\F_p$-complete space. A homomorphism $i \colon Y \to X$ of $p$-compact groups is a pointed map $Bi \colon BY \to BX$. Finally, if the homotopy fiber $X/Y$ of $Bi$ is $\F_p$-finite, then we say that $i$ is a monomorphism and that $Y \le_i X$ is a subgroup of $X$.
\begin{ex}\label{ex:liepcompact}
	If $G$ is a compact Lie group with $\pi_0G$ a finite $p$-group, then $G^{\wedge}_p$ is a $p$-compact group, using \cite[Lemma 2.1]{dwyerwilkerson_finloop}.
\end{ex}
\begin{lem}\label{lem:emtypepcompact}
	Let $X$ be a $p$-compact group. Then $BX$ is of EM-type and $\pi_1X$ is an abelian pro-$p$ group. 
\end{lem}
\begin{proof}
	By definition $BX$ is a connected space, and by \cite[Lemma 2.1]{dwyerwilkerson_finloop} $\pi_1BX \cong \pi_0X$ is a finite $p$-group. That $H^*(BX)$ is of finite type follows from the main result of \cite{dwyerwilkerson_finloop}. 

For the second part, we can assume that $X$ is connected since $X$ is a loop space and hence all its connected components have the same homotopy type. Let $T$ be a maximal torus of $X$, then $\pi_1X$ is always a quotient of $\pi_1T \cong \Z_p^r$, see for example \cite[Remark 1.3]{dwyerwilkerson_transfer}. Since $X$ is $\F_p$-finite, $\pi_1X$ is $p$-complete \cite[Proposition VI.5.4]{bousfield_homotopy_1972}, and the claim follows.
\end{proof}
 For the following, we recall that the $\F_p$-cohomological dimension of a space $X$, denoted $\dim_p(X)$, is the largest integer $i$ for which $H^i(X) \ne 0$ (with the convention that $\dim_p(X) = \infty$ if there is no such integer and $H^*(X) \neq 0$, and $-\infty$ if $H^*(X)$ vanishes). 

\begin{prop}[Dwyer--Greenlees--Iyengar]\label{thm:gorenstein_pcompact}
	Let $X$ be a $p$-compact group, then $C^*(BX) \to \F_p$ is orientable Gorenstein of shift $\dim_p(X)$. 
\end{prop}
\begin{proof}
	This is shown in \cite[Section 10.2]{dgi_duality}, except for identifying the shift. The proof relies on the fact that the graded ring $H_*(X)$ satisfies algebraic Poincar\'e duality of dimension $a$, which is exactly the shift. This implies that $a = \dim_p(X)$. Finally, it is automatically orientable, because \Cref{lem:emtypepcompact} shows that the conditions of \Cref{lem:orientable} are satisfied.
\end{proof}

This has the following  consequences. First, we consider the relative Gorenstein property for $p$-compact groups. We show that monomorphisms induce relative Gorenstein morphisms. Then we prove that the homotopy fiber of a monomorphism is a mod $p$  Poincar\'e duality space.

\begin{cor}\label{prop:relgorensteinpcompact}
	Let $X$ be a $p$-compact group and $Y \le_i X$ a subgroup, then $C^*(BX) \to C^*(BY)$ is relatively Gorenstein of shift $q=\dim_p(X) - \dim_p(Y)$, i.e.,
	\[
\Hom_{C^*(BX)}(C^*(BY),C^*(BX)) \simeq \Sigma^q C^*(BY). 
	\]
\end{cor}
\begin{proof}
	We need only verify the conditions of \Cref{thm:relgorensteinspaces}(1) for the morphism $Bi \colon BY \to BX$. We have already seen in \Cref{lem:emtypepcompact} that classifying spaces of $p$-compact groups are always of EM-type. By the definition of a monomorphism the homotopy fiber $F$ of $Bi$ has finite-dimensional cohomology. Since $p$-compact groups always satisfy Gorenstein duality (\Cref{thm:gorenstein_pcompact}), \Cref{thm:relgorensteinspaces}(1) applies with $q = \dim_p(X) - \dim_p(Y)$ as required. 
\end{proof}

\begin{rem}
	An alternative proof of \Cref{prop:relgorensteinpcompact} using equivariant homotopy was given in \cite[Theorem 6.4]{bchv}.
\end{rem}

\begin{cor}
 	Let $i \colon Y \to X$ be a monomorphism of $p$-compact groups. Then $C^*(X/Y) \to \F_p$ is Gorenstein of shift $\dim_p(Y) - \dim_p(X)$. Moreover, if $X/Y$ is connected, then $C^*(X/Y)$ satisfies Poincar\'e duality of shift $\dim_p(Y) - \dim_p(X)$, i.e., $X/Y$ is a mod $p$ Poincar\'e duality space. 
 \end{cor}
 \begin{proof}
  	We will apply \Cref{thm:ascent_spaces}(2). We have already seen that the classifying space of a $p$-compact group is of EM-type and satisfies Gorenstein duality. Thus \Cref{thm:ascent_spaces}(2) applies to show that $C^*(X/Y) \to \F_p$ is Gorenstein of shift $\dim_p(Y) - \dim_p(X)$.
	
	We claim that $\pi_1(X/Y)$ is a pro $p$-group. To see this, observe that the long exact sequence in homotopy takes the form
\[
\cdots \to \pi_1(X) \to \pi_1(X/Y) \to \pi_0(Y) \to \cdots
\]
From \Cref{lem:emtypepcompact} we see that $\pi_1(X)$ and $\pi_1(Y)$ are abelian pro-$p$ groups. Since $\pi_0(X)$ is a finite $p$-group, it follows that $\pi_1(X/Y)$ is a pro-$p$ group, and hence $C^*(X/Y)$ is automatically orientable by \Cref{lem:orientable}. By \Cref{cor:Gorenteincosmall} (which applies because $C^*(X/Y)$ is cosmall by \cite[Section 5.5]{dgi_duality}) $C^*(X/Y)$ satisfies Poincar\'e duality of shift $\dim_p(Y) - \dim_p(X)$. 
  \end{proof} 

Recall from \Cref{ex:liepcompact} that if $G$ is a compact Lie group with $\pi_0(G)$ a $p$-group, then $G^{\wedge}_p$ is a $p$-compact group. In particular, the $p$-completion $U(n)^\wedge_p $ is a $p$-compact group. In order to prove local Gorenstein duality, we will use ascent along unitary embeddings.  We will need the following \cite[Theorem 6.2]{cancas_finloop}, which relies on the classification of $p$-compact groups. 

\begin{thm}\label{thm:unitaryembeddingpcompact}
	Any $p$-compact group $X$ admits a monomorphism $X \to U(n)^{\wedge}_p$ for some $n > 0$. 
\end{thm}
 Then we deduce that any $p$-compact group satisfies local Gorenstein. 

\begin{thm}\label{thm:localgorensteindualitypcompact}
	Let $X$ be a $p$-compact group, then $C^*(BX)$ satisfies local Gorenstein duality of shift $\dim_p(X)$. 
\end{thm}
\begin{proof}
	Choose a monomorphism $i \colon BX \to BU(n)^{\wedge}_p$. The same argument as in \Cref{prop:relgorensteinpcompact} along with the fact that $C^*(BU(n)^{\wedge}_p)$ satisfies local Gorenstein duality (\Cref{ex:U(n)}) shows that the conditions of \Cref{thm:relgorensteinspaces}(2) are satisfied, so $C^*(BX)$ satisfies local Gorenstein duality of shift $\dim_p(X)$ as claimed.  
\end{proof}

\subsection{Compact Lie groups}

Let $G$ be a compact Lie group and continue to work with $\F_p$-coefficients. Following Benson and Greenlees \cite{BensonGreenlees1997Commutative}, given a $d$-dimensional real representation $V$ of $G$, we say that it is orientable (with respect to $\F_p$) if the action of $G$ on $H_d(S^V;\F_p)$ is trivial. The adjoint representation is orientable if for example $G$ is finite or connected. In this case, we have the following, see \cite[Section 10.3]{dgi_duality} and \cite{bg_localduality}.
\begin{thm}[Dwyer--Greenlees--Iyengar, Benson--Greenlees]
	If $G$ is a compact Lie group whose adjoint representation is orientable, then $C^*(BG) \to \F_p$ is Gorenstein of shift $\dim(G)$, and $C^*(BG)$ satisfies local Gorenstein duality of the same shift. 
\end{thm}
\begin{rem}
  One can also verify using the methods of this paper that for a compact Lie group as above, $C^*(BG)$ satisfies local Gorenstein duality.  Indeed, a unitary embedding $G \to U(n)$ satisfies the assumptions of \Cref{thm:relgorensteinspaces}(2), which gives the result. This is equivalent to the approach taken in \cite[Proposition 4.32]{bhv2}. The methods used in this paper are essentially a combination of those used by Benson and Greenlees in \cite{bg_localduality} (to determine the relative Gorenstein condition) and by Benson in \cite{benson_shortproof}, who gave a short proof in the case of a finite group $G$, using an ascent result which is generalized in \Cref{prop:bhv_ascent}. 
\end{rem}
More generally, $C^*(BG) \to \F_p$ is a `generalized' Gorenstein morphism, where we allow twists by an invertible element that may not be a suspension of $\F_p$, see \cite{greenlees_borel}. In fact, even in the case where the shift is a suspension of $\F_p$, it may not be by $\dim(G)$, as the following example from \cite{greenlees_borel} demonstrates.
\begin{ex}
	Consider the compact Lie group $O(2)$ and suppose that $p$ is odd. Then,
	\[
\Hom_{C^*(BO(2))}(\F_p,C^*(BO(2))) \simeq \Sigma^3\F_p,
	\]
	whereas $\dim(O(2)) = 1$. 
\end{ex}
One can explain this example in the following way: $\Omega(BO(2))^{\wedge}_p$ is a $p$-compact group with $\dim_p(\Omega(BO(2))^{\wedge}_p) = 3$, and then we can apply \Cref{thm:gorenstein_pcompact}. 

  In general the classifying space $BG$ of any compact Lie group is $p$-good since the fundamental group is finite, so that $C^*(BG^{\wedge}_p) \simeq C^*(BG)$, and moreover $\pi_1(BG^{\wedge}_p)$ is a finite $p$-group, see \cite[Theorem 7.3]{bg_localduality}. Thus, if $\Omega (BG^{\wedge}_p)$ is $\F_p$-finite, then $BG^{\wedge}_p$ is the classifying space of the $p$-compact group $\Omega (BG^{\wedge}_p)$.  Note that $\Omega (BG^{\wedge}_p)$ is not necessarily equivalent to $G^{\wedge}_p$, and in fact, that happens only if $\pi_0G$ a finite $p$-group.
 We record this in a definition. 

\begin{defn}[Ishiguro]
	A compact Lie group $G$ is said to be of $p$-compact type if $\Omega (BG^{\wedge}_p)$ is $\F_p$-finite. 
\end{defn}

%\ncomment{In Greenlees \cite[9.D]{greenlees_hi}, there is a notion of h-regularity: $X=BG^\wedge_p$ is g-regular iff $\pi_*(\End_R(k))$ is finite dimensional where $R=C^*(X)$. That is the same as $p$-compact type. He writes in page 24, $X=BG^\wedge_p$ is g-regular iff $\pi_0(G)$ is $p$-nilpotent.}

In particular, if $G$ is of $p$-compact type, then $\Omega (BG^{\wedge}_p)$ is a $p$-compact group.  We can then apply \Cref{thm:gorenstein_pcompact,thm:localgorensteindualitypcompact} to $\Omega (BG^{\wedge}_p)$ to deduce the following. 
\begin{thm}\label{thm:liepcompact}
	Let $G$ be a compact Lie group of $p$-compact type, then $C^*(BG) \to \F_p$ is Gorenstein of shift $\dim_p(\Omega (BG^{\wedge}_p))$ and $C^*(BG) \to \F_p$ satisfies local Gorenstein duality of the same shift. 
\end{thm}

\begin{ex}
	We return to the example of $O(2)$ at an odd prime. By a direct calculation with the Eilenberg--Moore spectral sequence, one can calculate that $\Omega (BO(2)^{\wedge}_p)$ is $\F_p$-finite and has $\dim_p(\Omega(BO(2)^{\wedge}_p) = 3$, as expected. More generally, at odd primes one has that $H^*(BO(2n)^{\wedge}_p) \cong \F_p[p_1,\ldots,p_n]$ with $|p_i| = 4i$. The Eilenberg--Moore spectral sequence shows that $H^*(\Omega (BO(2n)^{\wedge}_p)) \cong \Lambda_{\F_p}(t_1,\ldots,t_n)$, the exterior algebra on classes $t_i$ with $|t_i| = 4i-1$. This is $\F_p$-finite, and so $C^*(BO(2n))$ is of $p$-compact type, and hence satisfies local Gorenstein duality of shift $\dim_p(\Omega (BO(2n)^{\wedge}_p)) = \sum_{i=1}^n(4i-1)=  n(2n+1)$. In fact $BO(2n)^\wedge_p\simeq BSp(n)^\wedge_p$ at odd primes.
\end{ex}

It is an interesting open problem to find conditions on a compact Lie group $G$ so that $G$ is of $p$-compact type. Finite groups provide examples showing that it does not hold in general, for example, consider the case of $\Sigma_3$ at the prime 3. It is necessary, but not sufficient, that $G$ satisfies the following conditions (see \cite[Proposition 3.1]{Ishiguro2001Toral}):
\begin{enumerate}
	\item $\pi_0G$ is $p$-nilpotent
	\item $\pi_1((BG)^{\wedge}_p)$ is isomorphic to a $p$-Sylow subgroup of $\pi_0G$. 
\end{enumerate}
Under some more hypotheses, we have stronger results, for example if the connected component has rank 1, then $G$ is of $p$-compact type if $\pi_0G$ is $p$-nilpotent, see \cite[Theorem 2]{Ishiguro2001Classifying}.

\subsection{Local Gorenstein duality for rational spaces}
We now switch to the rational case, i.e., we take $k = \Q$. In this case, the conditions of \Cref{thm:relgorensteinspaces} become particularly easy to check given the fact that the algebraic Noether normalization can be lifted to spectra. 

\begin{thm}\label{thm:rational}
	Let $X$ be a simply connected rational space with Noetherian cohomology. Then, if $C^*(X;\Q) \to \Q$ is Gorenstein of shift $r$, then $C^*(X;\Q)$ satisfies local Gorenstein duality of shift $r$. 
\end{thm}
\begin{proof}
	The key observation is due to Greenlees--Hess--Shamir \cite[Proposition 3.2]{GreenleesHessShamir2013Complete}. By assumption on $X$ there exists a Noether normalization $R_* \cong \Q[x_1,\ldots,x_n]$ of $H^*(X;\Q)$, where the polynomial algebra is concentrated in even degrees. We can realize this polynomial subring via a map $X \to \prod_i K(\Q,2k_i)$, giving a map of ring spectra $R=C^*(\prod_i K(\Q,2k_i),\Q) \to C^*(X;\Q)$, which is finite by \cite[Lemma 10.2]{greenlees_hi}.

	We now observe that $R_*$ is a regular local ring, and in particular is Gorenstein. The universal coefficient spectral sequence
	\[
\Ext_{\pi_*R}^{\ast,\ast}(k,\pi_*R) \implies \pi_*\Hom_R(k,R)
	\]
	shows that $R \to k$ is Gorenstein of some shift $r$. We claim that $R$ satisfies local Gorenstein duality of the same shift. Indeed, since $R_*$ is a regular local ring, it is algebraically Gorenstein in the sense of \cite[Definition 4.5]{bhv2}, and hence satisfies local Gorenstein duality by \Cref{prop:4.7bhv2}. That the shifts are the same can be deduced from the collapsing of the local cohomology spectral sequence of Greenlees--Hess--Shamir \cite[Proposition A.2]{GreenleesHessShamir2013Complete}. 

By \cite[Section 7.B]{greenlees_hi} both $R$ and $C^*(X;\Q)$ are dc-complete. Thus, if $C^*(X;\Q) \to \Q$ is Gorenstein of shift $r$, then \Cref{thm:relgorensteinspaces} applies to show it is satisfies local Gorenstein duality of shift $r$. 
\end{proof}
\begin{ex}
The following example is taken from \cite[Example A.6]{GreenleesHessShamir2013Complete}. Identify, $\mathbb{C}P^{\infty}$ with $BS^1$, and then consider the inclusion $BS^1 \to BS^3$. This gives a map $\mathbb{C}P^{\infty} \times \mathbb{C} P^{\infty} \to BS^3 \times BS^3$. We let $X$ denote the fiber,  so that there is a fibration
	\[
S^3 \times S^3 \to X \to \mathbb{C}P^{\infty} \times \mathbb{C} P^{\infty}.
	\]
 The rational cohomology ring $H^*(X;\Q)$ is isomorphic to $\Q[u,v,p]/(u^2,uv,up,p^2)$ where $u,v,p$ have degrees $2$, $2$, and $5$. This ring is not Gorenstein, however, the map $C^*(X;\Q) \to \Q$ is Gorenstein of shift $-4$, and hence satisfies local Gorenstein duality of the same shift by \Cref{thm:rational}. 
\end{ex}

\section{Local Gorenstein duality for $p$-local compact groups}\label{sec:pcompact}
In this section we continue to write $C^*(X)$ and $H^*(X)$ where the coefficients are understood, unless any confusion is likely to arise. Recall that a compact Lie group gives rise to a $p$-compact group whenever $\pi_0G$ is a finite $p$-group. In order to capture the homotopy theory of compact Lie groups in general, Broto, Levi, and Oliver introduced the concept of a $p$-local compact group \cite{blo_fusion,blo_pcompact}. To motivate the definition, let $G$ be a finite $p$-group, $S$ a Sylow $p$-subgroup, and consider the category $\cal{F}_S(G)$ with objects subgroups of $S$ and morphisms $\Hom_{\cal{F}_S(G)}(P,Q) = \Hom_G(P,Q)$, those morphisms induced by subconjugation inside $G$. This is the fusion category of $G$ over $S$, and many results and concepts in group theory can be stated in terms of this category. The idea of $p$-local finite groups, and more generally $p$-local compact groups, is to generalize this where we are given only a finite $p$-group $S$ (or a discrete $p$-toral group), and we define a category from this, with similar properties to the category $\cal{F}_S(G)$. 

 In more detail, we fix a discrete $p$-toral group $S$, that is, a group that fits in an extension
\[
1 \to (\Z/p^{\infty})^r \to S \to \pi \to 1,
\]
where $\pi$ is a finite $p$-group. We then define a fusion system $\cal{F}$ on $S$ to be a category whose objects are the subgroups of $S$, and whose morphisms satisfy the following properties for $P,Q \le S$:
\begin{enumerate}
	\item $\Hom_S(P,Q) \subseteq \Hom_{\cal{F}}(P,Q) \subseteq \Inj(P,Q)$.
	\item Every morphism in $\cal{F}$ factors as an isomorphism followed by an inclusion. 
\end{enumerate}
In general, one is interested in fusion systems which are saturated, a technical condition defined in \cite[Definition 1.2]{blo_pcompact}. Given a fusion system on a discrete $p$-toral group $S$, Broto, Levi, and Oliver constructed another category, the centric linking system $\cal{L}$. One then defines a classifying space $|\cal{L}|^{\wedge}_p$, as the $p$-completion of the nerve of the category $\cal{L}$. In the case where $S$ is a finite $p$-group, it was shown by Chermak \cite{chermak_existence} that the data $(S,\cF)$ already uniquely determines the category $\cL$, while the general case was shown by Levi and Libman \cite{LeviLibman2015Existence}. 

\begin{defn}
	A $p$-local compact group $\cal{G} = (S,\cal{F})$ consists of a discrete $p$-toral group $S$ and $\cal{F}$ a fusion system on $S$. The classifying space $B\cal{G}$ is defined as the $p$-completion of the nerve of the centric linking systems $\cal{L}$ associated to $(S,\cal{F})$. If $S$ is a finite $p$-group, then $\cal{G}$ is called a $p$-local finite group. 
	\end{defn}
We point out that this classifying space comes with a canonical morphism $\theta_S \colon BS \to B\cal{G}$. 
\begin{lem}
	For a $p$-local compact group $\cal{G} = (S,\cal{F})$, the classifying space $B\cal{G}$ is of Eilenberg--Moore type. 
\end{lem}

\begin{proof}
	By \cite[Proposition 4.4]{blo_pcompact} the classifying space $B\cal{G}$ is $p$-complete, and $\pi_1B\cal{G}$ is a finite $p$-group. The cohomology ring is Noetherian \cite[Corollary 4.26]{bchv}, and hence of finite type. 
\end{proof}

The notion of monomorphism in this context involves a more general condition on the homotopy fiber of a map between classifying spaces than the one for $p$-compact groups introduced by Dwyer and Wilkerson, but it just specializes to that one when the spaces involved are $p$-compact groups.

\begin{defn}
	A morphism $f \colon X \to Y$ of connected spaces is called a homotopy monomorphism if the homotopy fiber $F$ is $B\Z/p$-local for every choice of basepoint in $F$, i.e., $\Map_*(B\Z/p,F)$ is contractible for all choices of basepoint in $F$. 
\end{defn}

\begin{rem}
When we are dealing with $p$-compact groups, \cite[Proposition 2.2]{cancas_finloop} shows that $f\colon BX\rightarrow BY$ is a monomorphism if and only if it is a homotopy monomorphism.
\end{rem}

\begin{prop}\label{prop:homotopymono}
	Let $\cal{G} = (S,\cal{F})$ and $\cal{H} = (P,\cal{E})$ be $p$-local compact groups. If $\phi \colon B\cal{G} \to B\cal{H}$ is a homotopy monomorphism, then $H^*(B\cal{G})$ is a finitely generated $H^*(B\cal{H})$-module via $B\phi^*$. 
\end{prop}
\begin{proof}
We recall that we have canonical maps $\theta_S \colon BS \to B\cal{G}$ and $\theta_{P} \colon BP \to B\cal{H}$.  Consider the composite $\phi\circ\theta_S\in [BS,B\cal{H}]$.  By \cite[Theorem 6.3 ]{blo_pcompact}, there is a bijection $[BS,B\cal{H}]\cong \Hom(S,P)/\sim$ where $\rho \sim \rho'$ if there is $f\in \Hom_{\cal{E}}(\rho(S),\rho'(S))$ such that $\rho'=f\circ \rho$. Then there exists a homomorphism $\rho \colon S \to P$, unique up to $\cal{H}$-conjugacy, making the following diagram commute
	\[
	\begin{tikzcd}
BS \arrow[r, "\theta_S"] \arrow[d, "B\rho"'] & B\cal{G} \arrow[d, "\phi"] \\
BP \arrow[r, "\theta_P"'] & B\cal{H}.
\end{tikzcd}
\]

We claim now that if $f$ is a homotopy monomorphism, then $\rho$ is a monomorphism. We first observe that $\theta_S$ and $\theta_P$ are homotopy monomorphisms, see the proof of \cite[Theorem 2.5]{cancas_finloop}. We also have that the composite of homotopy monomorphisms is a homotopy monomorphism, and so the diagram above implies that the composite $\theta_P \circ B\rho$ is a homotopy monomorphism. Now suppose that $\rho$ was not a monomorphism. Let $F$ denote the fiber of $\theta_P \circ B\rho$. Let $\sigma \in \Hom(\Z/p,S)$ be an injection into $\ker(\rho)$. This implies that $B \sigma \colon B\Z/p \to BS$ is null-homotopic in $B\cal{H}$, and hence $B\sigma$ lifts to a map $B\Z/p \to F$ which is not null-homotopic. This is a contradiction to $\theta_P \circ B{\rho}$ being a homotopy monomorphism, and so $\rho$ is a monomorphism as claimed. 

Taking cohomology, we obtain a commutative diagram 
	\[
	\begin{tikzcd}
H^*(B\cal{H}) \ar[r, "\theta_P^*"] \ar[d,"\phi^*",swap] & H^*(BP) \ar[d,"B\rho^*"] \\
H^*(B\cal{G}) \ar[r,"\theta_S^*",swap] & H^*(BS)
\end{tikzcd}
\]
By \cite[Cororally 4.20 and Proposition 5.5]{bchv} the cohomology rings in this diagram are Noetherian, and the horizontal morphisms are injections which exhibit the target as a finitely generated module over the source. We claim that in order to show that $H^*(B\cal{G})$ is a finitely generated $H^*(B\cal{H})$-module, it suffices to show that $H^*(BS)$ is finitely generated as $H^*(BP)$-module via $B\rho^*$. Indeed, suppose that $B\rho^*$ is finite, then $H^*(BS)$ is a finitely generated $H^*(B\cal{H})$-module. Because $H^*(B\cal{H})$ is Noetherian, it follows that $H^*(B\cal{G})$ is a finitely generated $H^*(B\cal{H})$-module, so $\phi^*$ is finite. 

%To see that $B\rho^*$ is finite, recall that any discrete $p$-toral group $S$ has a functorial closure, i.e., there is a $p$-compact toral group $\widetilde S$ (a $p$-compact group which is an extension of a $p$-compact torus and a finite $p$-group) and a morphism $f \colon BS \to B\widetilde S$ which is an $\F_p$-equivalence. 
To see that $B\rho^*$ is finite, consider the $p$-compact toral group $\widetilde S = \Omega ((BS)^{\wedge}_p)$ (a $p$-compact group which is an extension of a $p$-compact torus and a finite $p$-group) and the natural $\F_p$-equivalence $f \colon BS \to B\widetilde S$. By \cite[Proposition 3.4]{MollerNotbohm1994Centers}, the morphism $\widetilde \rho \colon \widetilde S \to \widetilde P$ is a monomorphism of $p$-compact groups since $\rho \colon S \to P$ is an algebraic monomorphism. This implies that $H^*(B\widetilde S)$ is a finitely generated $H^*(B\widetilde P)$-module via $B\widetilde\rho^*$ by \cite[Proposition 9.11]{dwyerwilkerson_finloop}. Therefore, $H^*(BS)$ is a finitely generated $H^*(BP)$-module via $B\rho^*$.
\end{proof}
We would like to identify conditions that ensure that, given a homotopy monomorphism $\phi \colon B\cal{G} \to B\cal{H}$ as above, the induced morphism $\phi^* \colon C^*(B\cal{H}) \to C^*(B\cal{G})$ is finite.  The next results follows from combining \cite[Lemma 3.4 and Lemma 3.5]{shamir_pcochains}.

\begin{lem}\label{lem:finitesmall}
	Let $f \colon Y \to X$ be a map between $p$-complete spaces with fundamental groups finite $p$-groups and let $F$ denote the fiber of $f$. If $H_*(\Omega X$) is finite dimensional and $H^*(Y)$ is finitely generated over $H^*(X)$ via the induced map, then $f^* \colon C^*(Y) \to C^*(X)$ is finite and $H^*(F)$ is finite-dimensional. 
\end{lem}
\begin{prop}\label{prop:relgorensteinspaces}
		Let $\cal{G} = (S,\cal{F})$ and $\cal{H} = (P,\cal{E})$ be $p$-local compact groups and $\phi \colon B\cal{G} \to B\cal{H}$ a homotopy monomorphism with homotopy fiber $F$, where $B\cal{H}$ is the classifying space of a $p$-compact group of dimension $h$.  Then:
		\begin{enumerate}
		\item $C^*(B\cal{G})$ is Gorenstein if and only if $H^*(F)$ is a Poincar\'e duality algebra.
		\item If $C^*(B\cal{G})$ is Gorenstein of shift $g$, then $\phi^*\colon C^*(B\cal{H}) \to C^*(B\cal{G})$ is relatively Gorenstein of shift $h-g$, and $C^*(B\cal{G})$ satisfies Gorenstein local duality of shift $g$. 
		\end{enumerate}
\end{prop}
\begin{proof} We have already seen that classifying spaces of $p$-local compact groups are always of EM-type, and they are $p$-complete by definition. 

The first point follows directly from \Cref{thm:ascent_spaces}, by using that $H^*(F)$ is finite (\Cref{lem:finitesmall}) and  \Cref{cor:Gorenteincosmall}.
	For the second point, as usual, we wish to apply \Cref{thm:relgorensteinspaces}. 
	By \Cref{prop:homotopymono,lem:finitesmall} we have that $f^* \colon C^*(B\cal{H}) \to C^*(B\cal{G})$ is finite and $H^*(F)$ is finite-dimensional, where $F$ is the fiber of $\phi$.  By \Cref{thm:localgorensteindualitypcompact}, $C^*(B\cal{H})$ satisfies local Gorenstein duality of shift $h$. Thus, if $C^*(B\cal{G})$ satisfies Gorenstein duality of shift $g$, then \Cref{thm:relgorensteinspaces} applies to give the result. 
\end{proof}

\begin{defn}
	Let $X$ be a connected $p$-complete space, then a unitary embedding of $X$ is a homotopy monomorphism $X \to BU(n)^{\wedge}_p$ for some $n>0$. A $p$-local compact group $\cal{G}$ is said to admit a unitary embedding if its classifying space $B\cal{G}$ does. 
\end{defn}
\begin{rem}
	Since classifying spaces of $p$-compact groups always admit unitary embeddings by \Cref{thm:unitaryembeddingpcompact}, any homotopy monomorphism as in \Cref{prop:relgorensteinspaces} gives rise to a unitary embedding of $B\cal{G}$.
\end{rem}
\begin{cor}\label{cor:orientableunitaryembedding}
	Let $\cal{G}$ be a $p$-local compact group with a unitary embedding $\phi \colon B\cal{G} \to BU(n)^{\wedge}_p$. The homotopy fiber $U(n)/\cal{G}$ is a Poincar\'{e} duality space of dimension $q$ if and only if $C^*(B\cal{G})$ is Gorenstein of shift $n^2+q$ and satisfies local Gorenstein duality of the same shift. 
\end{cor}
\begin{proof}
	The condition on the homotopy fiber ensures that $C^*(B\cal{G})$ is Gorenstein of shift $n^2+q$ by \Cref{thm:ascent_spaces}. 
\end{proof}

\begin{rem} 
Suppose that there exists a unitary embedding $\phi \colon B\cal{G} \to BU(n)^{\wedge}_p$ with Poincar\'e duality fiber. Then, any homotopy monomorphism $f \colon B\cal{G} \to B\cal{H}$ into the classifying space of a $p$-compact group will have the property that the homotopy fiber is a Poincar\'e duality space. Indeed, $C^*(B\cal{G})$ will be Gorenstein by \Cref{cor:orientableunitaryembedding}, and hence by \Cref{prop:relgorensteinspaces}(1) the homotopy fiber of $f$ will be a mod $p$ Poincar\'e duality space. For example, this holds for any other unitary embedding of $\cal{G}$.  In this situation, the Gorenstein shift can be interpreted as a notion of dimension for the $p$-local compact group.
\end{rem}

We do not know in general that $p$-local compact groups admit unitary embeddings, and even when they do we do not know that the homotopy fiber is a Poincar\'{e} duality space. However, in the case of a $p$-local finite group (i.e., when $S$ is a finite $p$-group), we do in fact know this is the case. 

\begin{cor}\label{cor:plocalfinitegorenstein}
	Any $p$-local finite group $\cal{G}$ satisfies local Gorenstein duality of shift $0$. 
\end{cor}
\begin{proof}
	By \cite[Theorem 6.2]{ccm_vbplocalfinite} there is a homotopy monomorphism $B\cal{G} \to BU(n)^{\wedge}_p$ for some $n>0$ and, moreover, $C^*(B\cal{G})$ is Gorenstein of shift 0, see \cite[Theorem 6.7]{ccm_vbplocalfinite}. Thus \Cref{prop:relgorensteinspaces} applies to show that $C^*(B\cal{G})$ satisfies local Gorenstein duality of shift 0. 
\end{proof}

\bibliography{duality}\bibliographystyle{alpha}
\end{document}